\documentclass[11pt, reqno]{amsart}

\usepackage[utf8]{inputenc}

\usepackage[colorlinks=true,pagebackref,hyperindex]{hyperref}
\usepackage{mathrsfs} 
\usepackage[all]{xy}
\usepackage{color}
\usepackage{amsfonts}
\usepackage{amscd}
\usepackage{amssymb,latexsym,amsmath,amscd, graphicx} 
\usepackage[makeroom]{cancel}
\usepackage{bbm} 
\usepackage{cancel}
\usepackage[normalem]{ulem}

\definecolor{darkgreen}{rgb}{0.0, 0.7, 0.0}

\definecolor{cyan}{cmyk}{1,0,0,0}

\newcommand{\cdg}{\color{darkgreen}}

\newcommand{\bdg}{\begin{dg}}
\newcommand{\edg}{\end{dg}}

\newtheorem{tm}{Theorem}[section]

\newtheorem{lm}[tm]{Lemma}
\newtheorem{lemma}[tm]{Lemma}

\newtheorem{rmk}[tm]{Remark}

\newtheorem{??}[tm]{Question}

\newtheorem{definition}[tm]{Definition}

\newtheorem{ass}[tm]{Assumption}

\newcommand{\ben}{\begin{enumerate}}
\newcommand{\een}{\end{enumerate}}
\newcommand{\bit}{\begin{itemize}}
\newcommand{\eit}{\end{itemize}}
\newcommand{\beq}{\begin{equation}}
\newcommand{\eeq}{\end{equation}}
\newcommand{\la}{\label}
\newcommand{\n}{\noindent}
\newcommand\ci{\cite}

\font\tenmsb=msbm10
\font\sevenmsb=msbm7
\font\fivemsb=msbm5
\newfam\msbfam
\textfont\msbfam=\tenmsb
\scriptfont\msbfam=\sevenmsb
\scriptscriptfont\msbfam=\fivemsb
\def\Bbb#1{{\fam\msbfam #1}}
\font\teneufm=eufm10
\font\seveneufm=eufm7
\font\fiveeufm=eufm5
\newfam\eufmfam
\textfont\eufmfam=\teneufm
\scriptfont\eufmfam=\seveneufm
\scriptscriptfont\eufmfam=\fiveeufm

\newcommand{\lorw}{\longrightarrow}

\newcommand\rat{{\Bbb Q}}

\newcommand\oql{\overline{\Bbb Q}_\ell}
\newcommand\comp{{\Bbb C}}

\newcommand\pn[1]{{\Bbb P}^{#1}}

\newcommand\e{\epsilon}

\newcommand{\w}[1]{\widetilde{#1}}

\newcommand{\ov}[1]{\overline{#1}}

\newcommand{\ms}[1]{\mathscr{#1}}

\newcommand{\field}{k} 
\newcommand{\Gm}{\mathbb G_m}
\newcommand{\FC}{C^{(1)}} 

\newcommand{\A}{A} 

\newcommand{\ra}{\rightarrow} 
\newcommand{\xra}{\xrightarrow}
\newcommand{\ah}{\alpha}
\newcommand{\gam}{\gamma}

\newcommand{\Z}{\mathbb{Z}}

\newcommand{\sig}{\sigma}

\newcommand{\lam}{\lambda}

\usepackage[english]{babel}

\newcommand{\lb}{\langle}
\newcommand{\rb}{\rangle}

\newcommand{\hra}{\hookrightarrow}

\newcommand{\Proj}{\mathbb{P}}

\usepackage{tcolorbox}
\usepackage{tikz-cd}

\usepackage{indentfirst}
\usepackage[shortlabels]{enumitem}
\usepackage[all]{xy}

\newcommand{\mA}{\mathbb{A}}
\newcommand{\Gg}{\mathbb{G}}
\newcommand{\Bbase}{B}
\newcommand{\GB}{\Gg_{m,\Bbase}}
\newcommand{\oxb}{\omega_{X^{(\Bbase)}/\Bbase}}
\usepackage{dirtytalk}
\newcommand{\Spec}{\text{Spec}}

\title[Projective Completion of Moduli of $t$-Connections]{Projective completion of   moduli of $t$-connections on curves in positive and mixed characteristic}
\author{
Mark Andrea A.  de Cataldo and Siqing Zhang
}
\address{Mark Andrea A.  de Cataldo, Stony Brook University, mark.decataldo@stonybrook.edu}
\address{Siqing Zhang, Stony Brook University, siqing.zhang@stonybrook.edu}

\begin{document}

\maketitle

\begin{abstract} 
We generalize a compactification technique due to C. Simpson in the context of $\Gm$-actions over the ground  field of complex numbers, to the case of a universally Japanese base ring. 
We complement this generalized compactification technique so that it
can sometimes yield projectivity results for these compactifications. We apply these projectivity results to the Hodge, de Rham and Dolbeault moduli spaces
for curves, with special regards to ground fields of positive characteristic. 
\end{abstract}

\tableofcontents

\section{Introduction}\la{intro}

C. Simpson has introduced a compactification technique (\ci[\S11]{si naf})
which he then applied to compactify the moduli space of flat connections over a curve defined  over the complex numbers.
This paper has grown from the need  in \ci{de-szh naht} to    generalize 
this compactification technique so that it leads to compactifications
of the moduli spaces that appear in the Non Abelian Hodge Theory of a curve defined over an algebraically closed field of arbitrary characteristic, or over  a discrete valuation ring, 
possibly of mixed characteristic. 
For more details on these moduli spaces $M_{Hod}$ of $t$-connections, $M_{dR}$ of connections, and $M_{Dol}$
of Higgs bundles,  see \S\ref{rt51}.


 Let us discuss a little Simpson's  compactification technique leading to the compactification of the moduli space
 $M_{dR}$
 of connections on a curve $C$ over the complex numbers.  First, he constructs the moduli space of $t$-connections as a $\Gm$-equivariant morphism $\tau: M_{Hod} \to \mathbb A^1$, where a $t$-connection  on $C$ is sent to the scalar value $t$.
 The action is: multiply a $t$-connection by a non-zero scalar.
 For $t=0$, we have the Dolbeault moduli space $M_{Dol}$ of Higgs bundles, and for $t=1$ we have the de Rham moduli space of connections. There is the Hitchin morphism $h_{Dol}: M_{Dol}\to A$ ($A$ a suitable affine space parametrizing
 spectral curves for $C$ in the cotangent bundle of $C$). The fiber $N_{Dol}$ over the point $o\in A$ corresponding to the spectral curve $rC$ --rank times the zero section-- is   compact and is also the set of points in $M_{Hod}$
 that admit infinity limits for the $\Gm$-action.  Simspson sets $\ov{M_{dR}}:= (M_{Hod} \setminus N_{Dol})/\Gm;$ this way the boundary $\ov{M_{dR}} \setminus M_{dR}= (M_{Dol}\setminus N_{Dol})/\Gm.$ This definition is simple-minded. On the other hand, the proof that the quotient exists as a separated proper
 scheme over $\comp$ is quite clever and intricate (it does not use the methods from D. Mumford's Geometric Invariant Theory).  Simpson proves a more general result (\ci[\S11]{si naf}), where one takes the quotient by $\Gm$
 of  a suitable  $\Gm$-variety $U/S$ over a complex variety $S$ endowed with the trivial $\Gm$-action.
This is what we mean by Simpson's 
 compactifiction techinque. 
 The application to the compactification of $M_{dR}$  is a special case;  see  diagram (\ref{act hh}) in the proof of Theorem \ref{cpt tm hod}.
 
 The first set of results of this paper are stated in 
 \S\ref{cpt proj rez} and  are proved in  the lengthy and technical  \S\ref{proof ugm proper}.
 We generalize Simpson's compactification technique in Theorems \ref{ugm proper} and \ref{ugm proper 2}.
  In short,  the set-up $U/S/\comp$ above, is replaced  by one of the form $U/S/B/J$, where $B$ is a base scheme
  over a universally Japanese ring $J$, and the multiplicative group  acting is ${\mathbb G}_{m,B}.$ 
  This level of generality seems to be the natural one in view
 of A. Langer's results yielding, as special cases  of  his \ci[Tm. 1.1]{la-2014},  the moduli spaces we work with
 for families of curves over a base defined over such a ring.  
  This covers the case of discrete valuation rings, which is of interest in \ci{de-szh naht}. 
  We complement these results with the compactification and projectivity criteria in Theorem \ref{proj tm};
  here one works with an equivariant morphism $U/S \to U'/S$, and this is useful in our applications,
  as the moduli spaces we work with do carry such morphisms,  such as the Hitchin morphism $h_{Dol}$ seen above, and we want to compactifiy domain and target,
  while  keeping track of the morphism.

The second set of results are compactification results for the moduli spaces $M_{Hod}, M_{dR}$ and $M_{Dol}$.
Recall (\S\ref{rt51}) that we have natural morphisms
exiting these moduli spaces: the  proper Hitchin morphism $h_{Dol}$;
the structural morphism $\tau_{Hod}: M_{Hod} \to \mathbb A^1$.  In positive characteristic, we also have: 
the Hodge-Hitchin morphism $h_{Hod} : M_{Hod} \to A'\times \mathbb A^1$ (here $A'$ is a suitable affine space
parameterizing the spectral curves for the Frobenius twist of the curve);  the de Rham-Hitchin morphism
$h_{dR}: M_{dR} \to A'$.
  Our compactification results for these moduli spaces and the associated morphisms
  are stated in \S\ref{rt51} and proved in \S\ref{ss main rz}, as an application of 
Theorem \ref{proj tm}:
Theorem \ref{comphodj} ($M_{Hod}$); Theorem \ref{cpt tm hod} (positive characteristic, $M_{Hod}$ and $h_{Hod}$); 
Theorem \ref{Jcpt tm dr} ($M_{dR}$); Theorem \ref{cpt tm dr} (positive characteristic, $M_{dR}$ and $h_{dR}$);
Theorem \ref{Jcpt tm dol} ($M_{Dol}$); Theorem \ref{cpt tm dr} (positive characteristic, $M_{Dol}$ in relation
to $M_{Hod}$).

In fact, we also prove projectivity  results concerning  the aforelisted natural morphisms
exiting these moduli spaces. 
We prove, using the known fact that the Hitchin morphism $h_{Dol}$ is proper,    that the  morphisms 
$\ov{\tau_{Hod}}: \ov{M_{Hod}} \to \mathbb A^1$, $h_{Hod}$ and $h_{dR}$ are proper, in fact projective.

 The properness of $h_{dR}$ has been  proved by M. Groechenig \ci{gr-2016}, who deduces it from the properness of the Hitchin morphism. 
The Hodge-Hitchin morphism $h_{Hod}$ 
 has been introduced 
by Y. Laszlo and C. Pauly, who proved  (\ci[Pr. 5.1]{la-pa}) that it is proper when restricted over  $o_{A'} \times \mathbb A^1,$  where $o_{A'}$ is the ``origin" of $A'$ ($t$-connections with nilpotent $p$-curvature).
A. Langer 's \ci[statement at the top of p. 531 and Th. 5.1]{la-2014} implies  that $h_{Hod}$ is proper;
see Remark \ref{rmk ala}.
In either case, 
one applies a variant of the Langton technique to a related Hitchin morphism, and deduces from it the desired conclusion.
The proof we offer is via the compactification theorems we prove, but also relies
on Langer's  Langton-type result \ci[Th. 5.1]{la-2014}.
 
 The purpose of the Appendix \S\ref{appdx} is stated in   \S\ref{intro appdx}: in short, one wants to extend, under favorable circumstances,  the techniques and results in \ci{de-2021} concerning specialization morphisms
 in cohomology, from a situation over the complex numbers,  to the one over a  discrete valuation base ring. 
 This entails making sure that:
 we have suitable compactifications (this is achieved by the compactifications in \S\ref{rt51}); we have
 the correct formalism of perverse sheaves for schemes over a discrete valuation ring (this is confirmed in \S\ref{recdvr});
 we  carefully revisit \ci[\S4]{de-2021}  and make sure that some potential issues due to positive or mixed characteristic are ironed out, at least  under favorable circumstances (this is done in  the technical  \S\ref{a cpt sp}). 

 The results of this Appendix  \S\ref{appdx},  which relies heavily on the results in \S\ref{rt51}, 
are used in \ci{de-szh naht}.

{\bf Acknowledgments.}  We thank the referee for the very good suggestions.
We thank  Dan Abramovich,  Barghav Bhatt, H\'el\`ene Esnault, Camilla Felisetti,  
Michael Groechenig,  Jochen Heinloth, Annette Huber, Luc Illusie, Adrian Langer, Yifeng Liu, Davesh Maulik,
Mirko Mauri, Sophie Morel, David Rydh,   Junliang Shen,  Ravi Vakil, Angelo Vistoli,  and Weizhe Zheng.
for useful conversations.
M.A. de Cataldo is partially supported by NSF grants DMS 1901975 and by a Simons Fellowship in Mathematics.
S. Zhang is partially supported by NSF grant DMS  1901975.

\section{Statement of the main results}\la{rzze}

\subsection{Compactification and projectivity results}\la{cpt proj rez}$\;$

In order to prove the compactification Theorems in \S\ref{rt51} concerning Hodge, de Rham and Dolbeault moduli spaces associated with
curves over some suitable base schemes, we first need to  prove
the Projectivity Theorem III  \ref{proj tm}. In turn, to prove this latter result, we need to prove the Compactification Theorem
II \ref{ugm proper 2}, which is a direct consequence of the Compactification Theorem I \ref{ugm proper}, the proof of which takes the bulk of this paper.

The Compactifiction Theorems I \ref{ugm proper} and II  \ref {ugm proper 2} together generalizes Simpson's Compactification  
technique \ci[Thm.\;11.1,\;11.2]{si naf}, which is stated and proved by C. Simposn over the field of complex numbers, to the case over a base scheme as in Assumption \ref{assumption p}.

The Projectivity Theorem III  \ref{proj tm} is the arbitrary characteristic counterpart to 
 \ci[Prop. 3.2.2]{de cpt}. In fact, by using the notation of Theorem \ref{proj tm}, this same theorem
 replaces the assumption that $Z\to Z'$ is proper, with the weaker assumption that $U$ is the pre-image of $U'$.
 This improvement, coupled with auxiliary properness results, affords proofs  of properness and of projectivity of certain morphisms and objects  arising in Non Abelian Hodge Theory; see \S\ref{rt51} and their proofs in \S\ref{ss main rz}.

The Compactification Theorem I \ref{ugm proper} can also be viewed as a partial  generalization (replace the ground field with the base variety $S$)
 of the main theorem in \cite[p.11]{bi-sw} to the relative case. For a comparison between Theorem \ref{ugm proper} and the main theorem of \cite[p.11]{bi-sw}), see Remark \ref{relation with bi-sw}.

 Let us  introduce  the setup for the Compactification Theorem I \ref{ugm proper}. 
 This setup is similar to the one in \cite[p.\;44]{si naf}; the main difference
is that  we work over a base scheme $B$  over a universally Japanese ring $J$ \cite[\href{https://stacks.math.columbia.edu/tag/032E}{032E}]{stacks}), while Simpson works over the complex numbers (i.e. $J=\comp$).  (Added in revision: the paper \ci{La-2021} allows to merely assume $B$ to be Noetherian and drop $J$.)

 \begin{ass}[\bf Setup for the Compactification Theorem I \ref{ugm proper}]\la{assumption p}$\;$

\n
The following assumptions  concerning schemes $X/S/B/J$ remain in vigour up to and including Theorem \ref{ugm proper}. 
Let $J$ be a universally Japanese ring.
Let $B$ and $S$ be noetherian schemes.
Let $S\ra B\ra J$ be separated morphims of finite type.
Assume that $S$ admits an invertible sheaf that is ample, and an invertible sheaf that is ample relative to $B$;
this ensures that there is a $\Bbase$-morphism that is a locally closed embedding of $S$ into $\Proj^N_{\Bbase}$ for some $N>0$.
Let $X\ra S$ be a projective morphism.
Let $\GB:=\Gm\times_{\Z}B$.
Let $\mu: \GB \times_{B} X\ra X$ be a $\GB$-action on $X$ covering the trivial $\GB$-action on $S$. 
Assume that $X$ admits a $\GB$-linearized ample line bundle.
 \end{ass}

Our next goal is to state Theorem \ref{ugm proper} and, to this end, we need some preparation.

{\bf Limit points, fixed points.}
Let us recall the definition of limits of a point in $X$ under the $\GB$-action $\mu$.
Let $T$ be a $B$-scheme.
Let $x\in X(T)$. 
Let $\mu_x$ be the orbit morphism defined by the following compositum:
\begin{equation}
    \mu_x: \Gg_{m,T}:=\GB\times_B T\xra{id\times x}\GB\times_{B} X\xra{\mu} X.
\end{equation}

If $\mu_x$ extends to a morphism $\w{\mu_x}:\mA^1_T\ra X$, then this extension is unique. In this case, we say that $\lim_{t\ra 0} t\cdot x$ exists and we set it to be the restriction of 
$\w{\mu_x}$ to $0_T\subset \mA^1_T$. Clearly,  $\lim_{t\ra 0} t\cdot x$ is an $T$-point of $X$.
Similarly, if $\mu_x$ extends to a morphism $\w{\mu_x}': \Proj^1_T\setminus 0_T\ra X_T$, then this extension is unique, and we set $\lim_{t\ra \infty} t\cdot x$ to be the restriction of $\w{\mu_x}'$ to $\infty_T\subset \Proj^1_T\setminus 0_T$, which is also an $T$-point of $X$.
By \cite[XII Cor.\;9.8]{sga3},
there exists the closed subscheme $V\subset X$ of fixed points for the $\GB$ action.

{\bf A partial order.}
Let us  introduce  a partial order on the Zariski points of $V$ as in \cite[p.\;44]{si naf} via the following definition.
The weight argument as in \cite[p.\;44]{si naf}
shows that the upcoming relation  $\le$ indeed defines a partial order on the Zariski points of $V$.

\begin{definition}[\bf The Partial Order $\le$ on the Zariski points of $V$]
\label{def order}
Let $u$ and $v$ be two Zariski points of $V$. Define a relation $\le$ as follows:   $u\le v$ if there exists a finite sequence of Zariski points $x_1,...,x_m$ of $X$ such that $\lim_{t\ra 0} t\cdot x_1=u$, $\forall 1\le l \le m-1$, $\lim_{t\ra \infty} t\cdot x_l=\lim_{t\ra 0} x_{l+1}$, and  $\lim_{t\ra \infty} t\cdot x_m =v$. 

If $u\le v$, then we say that  $u$ is more zero than $v$, and $v$ is more infinity then $u$.
\end{definition}

\begin{definition}[\bf Partitions $V=V^+\cup V^-$]\label{define V+}
We consider partitions of $V$  with the following properties.
Let $V_+$ and $V_-$ be two disjoint closed and open subschemes of $V$ with the property that $V=V_+\cup V_-$.
In addition, we require the following: if $u$ is more zero than a point in $V_+$, then $u\in V_+$; if $v$ is more infinity than a point in $V_-$, then $v\in V_-$.
\end{definition}

{\bf Concentrators.}
Let $Z$ be any $\Gm$-stable closed subscheme of $X$.
Let the $0$-concentrator functor $\Phi_0$ be the subfunctor of $X$ 
such that for any $B$-scheme $T$, a $T$-point $x$ of $X$ is in $\Phi_0(T)$ 
if and only if $\lim_{t\ra 0} t\cdot x$ exists and lies in $Z(T)$.
By \cite[\S4.5]{he-1981}, we see that $\Phi_0$ is represented by a scheme $X_0(Z)$ with a morphism $X_0(Z)\ra X$ that is locally over $X_0(Z)$ a locally closed immersion (note that $X_0(Z)\ra X$ may not be a locally closed immersion, see \cite[\S4.6]{he-1981}). 
Similarly, we can define the $\infty$-concentrator $\Phi_{\infty}$ and the scheme  morphism $X_{\infty}(Z)\to X$.

\begin{definition}[\bf Set Theoretic Partition $X=Y_+\cup Y_-\cup U$]
\label{define Y and U}
We fix a partition $V=V^+ \cup V^-$ of the fixed point set  as in Definition \ref{define V+}.
We define $Y_+$ to be the set theoretic image of $X_{\infty}(V_+)\ra X$, and define $Y_-$ to be the set theoretic image of $X_0(V_-)\ra X$.
We define the set $U:=X\setminus (Y_+\cup Y_-)$.
\end{definition}

\begin{rmk}\la{ta disj}
Let us show that the sets $Y_\pm$ in Definition \ref{define Y and U} are disjoint.
If there were  $x\in Y_+\cap Y_-$, then $x$ would be  more zero than a point $u\in V_+$, and more infinity than a point $v\in V_-$.
We would then have that $u$ is more zero than $v\in V_-$. By Definition \ref{define V+}, we would have that $u\in V_+\cap V_-$, contradicting that $V_+\cap V_-=\emptyset$. 
\end{rmk}

Recall that a uniform (resp. universal) geometric quotient $A\to B$ is a geometric quotient whose formation commutes with flat (resp. arbitrary) base change $B'\to B.$

\begin{tm}[\bf Compactification Theorem I]
\label{ugm proper}
Assumption \ref{assumption p} on $X/S/B/J$ are  in vigour.
Fix a partition of the fixed point set $V=V^+ \cup V^-$ as in Definition \ref{define V+} and let
$X=Y_+\cup Y_-\cup U$ be the corresponding   set theoretic partition  as in Definition \ref{define Y and U}. We have that 
\ben
\item
Both $Y_+$ and $Y_-$ are closed inside $X$.
\item
The uniform geometric quotient $U\to U/\GB$ exists, with $U/\GB$  an $S$-scheme.
\item
The morphism $U/\GB\ra S$ is universally closed.
\item
The morphism $U/\GB\ra S$ is separated, thus, in view of (3) above, proper.
\een
\end{tm}
\begin{proof}
The proof occupies the whole of \S\ref{proof ugm proper}.
\end{proof}

\begin{tm}[\bf  Compactification Theorem II]\la{ugm proper 2}
Assumption \ref{assumption p}  on $S/B/J$ are  in vigour.
Suppose $Z/S$ is an $S$-scheme with a $\GB$-action that is compatible with the trivial $\GB$-action on $S$.
Assume that there is a $\GB$-linearized relatively ample line bundle on $Z/S$. 
Suppose that the fixed point set $W \subseteq Z$ is proper over $S$, and that for any $z \in Z$
the limit  $\lim_{t \to 0} t \cdot z$  exists in $W$. 
Let $U \subseteq  Z$ be the subset of points $z$ such that the limit
$\lim_{t \to \infty} t \cdot z$ does not exist in $Z$. 
Then $U$ is open and there exists a uniform geometric quotient
$U\to U/\GB$ by the action of $\GB$. This geometric quotient is separated and proper over $S$.
\end{tm}
\begin{proof}
This follows from Theorem \ref{ugm proper} in the same way in which \ci[Thm. 11.2]{si naf}
follows from \ci[Thm. 11.1]{si naf}. We only  reproduce some of the highlights of the  proof. 
Use the $\GB$-linearized relatively ample line bundle on $Z/S$ to embed $Z/S$  $\GB$-equivariantly into some $\pn{N}_S$ as a locally closed subvariety. Take the closure and call it $X/S.$
Let $V \subseteq X$ be the fixed point set. Define   $V_+:=W$ to be the fixed point set in $Z.$ Let $V_-:= V \cap (X\setminus Z).$
The rest of the proof consists of showing that $V_+$ and $V_-$ have the desired properties, and that $U$, as it is defined in the statement of this theorem, is indeed
$X \setminus (Y_+\cup Y_-).$ At this juncture, one applies Theorem \ref{ugm proper}.
\end{proof}


{\bf Setup for the Projectivity Theorem III \ref{proj tm}.}
Assumption \ref{assumption p}  on $S/B/J$ are  in vigour.
Let $Z$ and $Z'$ be  varieties over $S,$
endowed with a $\GB$-action covering the trivial $\GB$-action over $S$, so that the structural morphisms $Z,Z'\to S$ are $\GB$-equivariant.
Let $Z\to Z'$ be a $\GB$-equivariant $S$-morphism.

\begin{tm}[\bf Projectivity Theorem III]\la{proj tm} 
Let $U\subseteq Z$ ($U' \subseteq Z'$, resp.)  be the subset
such that the $\infty$-limits do not exist. 
Assume that 
\ben
\item[(a)]
$Z/S$ and $Z'/S$ carry  relatively  ample line bundles admitting $\GB$-linearizations. 
\item[(b)]
The fixed point set $V \subseteq Z$ is proper over $S.$ 
\item[(c)]
The  $0$-limits exist in $Z$.
\item[(d)]
At least one of the following two conditions is met
\ben
\item[(i)] the $\GB$-equivariant $S$-morphism $Z\to Z'$ is surjective; 
\item[(ii)]  the fixed point set 
${V'} \subseteq Z'$ is proper over $S$ 
and the $0$-limits exist in $Z'.$

\een 
\item[(e)]
 $U$ is the preimage of $U'$ (this is automatic if $Z \to Z'$ is proper).
 \een
Then: 

\ben
\item
$U$  ($U'$, resp,) is open in $Z$ ($Z'$, resp.);  

\item
The morphism $U\to U'$ descends to a proper $S$-morphism  $U/\GB \to U'/\GB$  between the  geometric quotients, both of which are proper and separated over $S;$

\item

\ben
\item
the descended morphism 
$U/\GB  \to U'/\GB$ is projective; 
\item
if, in addition, $(U'/\GB)/S$ is also  projective, then $(U/\GB)/S$ is projective.
\een

\een
\end{tm}
\begin{proof}
The proof is identical to the one in \ci[Prop. 3.2.2]{de cpt}. Note that in loc.cit. the current assumption (e) is replaced by the assumption
that $Z/Z'$ is proper; the proof in loc.cit. works with the current assumption in place of the properness assumption on $Z/Z'$.

For the reader's convenience, we discuss briefly  the structure of the proof.
Parts (1,2)  can be proved along the same lines of the proof of  Theorem \ref{ugm proper 2}.
We simply note the following: the assumption  (d.i) on surjectivity  implies easily the assumption (d.ii) on the properness of the
fixed point set and the existence of $0$-limits.
One applies the Compactification Theorem II \ref{ugm proper 2} to $Z$ and to $Z'$ to find the uniform geometric
quotients $U/\GB$ and $U'/\GB$.  The descended morphism  $U/\GB \to U'/\GB$ between the uniform geometric quotients arises from the $\GB$-equivariance of the morphism $U\to U'.$
The properness and separateness over $S$ of these quotients follow from Theorem \ref{ugm proper 2}. 
The properness of the descended morphism follows from
the properness of $(U/\GB)/S$.  

What needs proof is part (3).
Part (3) is proved in \ci[Prop.\;3.2.2]{de cpt} (the set-up there is the one of characteristic zero, but the proof works for arbitrary base scheme $B$).

The key part is (3.a):

The proof in \ci[Prop.\;3.2.2]{de cpt} relies on Kempf's Descent Lemma \cite[Thm.\;2.3]{dr-na}, which is stated over fields of characteristic  zero.
A generalization of Kempf's Descent Lemma to the case over a more general scheme
can be found in \cite[Thm.\;10.3]{alper} and \cite[Thm.\;1.3.(iii)]{rydh-2020}.

To apply \cite[Thm.\;10.3]{alper}, we need to show that:
(i) The uniform geometric quotient $U/\GB$ given by Theorem \ref{ugm proper} is a good and tame moduli space for the quotient stack $[U/\GB]$;
(ii) some tensor power of the $\GB$-linearized ample line bundle on $U$ has trivial stabilizer action at closed points, in the sense of \cite[Def.\;10.1]{alper}.

(i) follows from Remark \ref{tamegood}.
(ii) follows from the fact that the stabilizers of the closed points of $U$ are finite subgroup schemes of $\Gm$ over the corresponding residue fields, and that for a finite group scheme $G$ of order $n$ over a field, the $n$-th power morphism $g\mapsto g^n: G\ra G$ is the identity morphism, see \cite[Prop.\;11.32]{milne}.

Then  one proves the $(U/\GB)/(U'/\GB)$-ampleness
of the descended line bundle by observing that it is ample on the fibers of the proper  morphism $(U/\GB)/(U'/\GB)$.
\end{proof}

\begin{rmk}[\bf Comparison of  Theorems \ref{ugm proper} and \ref{ugm proper 2} with \ci{si naf}]\label{comparison with si}$\;$ 

\n
The Compactification Theorems I,II \ref{ugm proper}, 
\ref{ugm proper 2}
are  stated and proved in \ci[Thm. 11.1, 11.2]{si naf} over the complex numbers, where loc.cit. Thm. 11.2 is a corollary to loc. cit Thm. 11.1,
the same way the Compactification Theorem  II \ref{ugm proper 2} follows from the Compactification Theorem I  \ref{ugm proper}.

 As it is observed in \ci[\S3.2]{de cpt},  C.Simpson's \ci[Thm.\;11.1, 11.2]{si naf}
are missing a seemingly necessary hypothesis on the existence of a $\Gm$-linearized $X/S$-ample line bundle.
This minor point out of the way, all the  necessary ideas are clearly stated by C. Simpson in \ci[\S11]{si naf}. We felt that some details were  present only in implicit form, and then only within 
a characteristic zero setup.
Since in this paper we need these results also over a base, we felt the need to write a detailed proof 
of the Compactification Theorem  I \ref{ugm proper}.  
Again, all the ideas in the proof of the Compactification Theorems I, II are due to C. Simpson.
\end{rmk}

\begin{rmk}[\bf Comparison of Theorem  \ref{ugm proper}  with \ci{bi-sw}]\label{relation with bi-sw}$\;$

\n
When $S=\Spec(k)$, the set $U$ of Theorem \ref{ugm proper} is called a sectional set in \cite[Def.\;1.2]{bi-sw};
the same paper also considers semi-sectional sets. 
The most obvious difference between sectional and semi-sectional sets is that, unlike a sectional set,  a semi-sectional set may contain some, but not arbitrary, $\Gm$-fixed point.
For a semi-sectional set $U'$, it is also proved in \cite[Thm.\;3.1]{bi-sw} that $U'/\Gm$ is a semi-geometric quotient.
The difference between a geometric and a semi-geometric quotient is that a point in a semi-geometric quotient may corresponds to multiple orbits. In this paper, we do not consider semi-sectional sets.

The proof in \cite[Thm.\;3.1]{bi-sw} is obtained by first establishing what are the possible configurations of  fixed-point sets 
for actions of $\Gm$ on projective spaces $\Proj^N_\field$. 
One equivariantly embeds $X$ in some $\pn{N}_\field$ by using the ample 
$\Gm$-linearized line bundle.
This is followed by an inductive analysis, and here we summarize very roughly, of how the fixed-point set  on $X$ is related to 
the fixed-point set  of the ambient $\Proj^N_k$. 
In the relative case, it is not clear to us how to piece together the possible global configurations of the fixed-point sets  of $(\Proj^N_S,X)$
fiber-by-fiber over $S$. Therefore, it is not clear to us how to modify the proof in \cite[Thm.\;3.1]{bi-sw} to make it work in the relative case  over $S$ we 
are working with.
\end{rmk}

\begin{rmk}[\bf Comparison of  Theorem \ref{proj tm} with \ci{de cpt}]\label{comparison with dec}$\;$ 

(a)
The items (1), (3), and (4) of Projectivity Theorem III \ref{proj tm} are essentially borrowed from \ci[Prop. 3.2.2]{de cpt}.
loc.cit.  is stated and proved  in characteristic zero, but,
once the Compactification Theorems I, II, \ref{ugm proper} and \ref{ugm proper 2} are in place, the proof carries over to arbitrary characteristic.

(b)
Moreover, we remove from \ci[Prop. 3.2.2]{de cpt} the hypothesis that $Z\to Z'$ is proper, and we replace it with the weaker hypothesis
that $U$ is the preimage of $U'.$ Observe    that the preimage of $U'$ sits inside $U$ automatically. If one assumes that  $Z\to Z'$
is proper, then one shows that  the preimage of $U'$ is $U.$ 
In all the  applications of the Projectivity Theorem III  that we provide in this paper,
i.e. in \S\ref{rt51},  the sets $U$ and $U'$ are constructed, and the preimage of $U'$ is verified to be $U$ by inspection of the construction.
In all such applications, we have that $U'/U$ is proper. 
We ignore if the assumptions of Theorem \ref{proj tm} imply that $U/U'$ must be proper.
\end{rmk}

\subsection{Applications to Projectivity in Non Abelian Hodge Theory}\la{rt51}$\:$

We introduce the setup for our main projectivity results, Theorems \ref{cpt tm hod}  (Hodge/$t$-connections), \ref{cpt tm dr} (de Rham/flat connections) 
and \ref{cpt tm dol} (Dolbeault/Higgs bundles).

The context is the one of moduli spaces of
$t$-connections on a curve,  which is a kind of umbrella covering, in some sense, Higgs bundles and flat connections.
The notion of $t$-connections was introduced and studied  by C. Simpson \ci{si naf} over the complex numbers,
 and by  Y. Lazslo and C. Pauly  \ci{la-pa} in positive characteristic. 

{\bf Smooth Curves.}
Let $\Bbase$ be a noetherian scheme that is finite type over a universally Janpanese ring $J$.
Let $\pi: C\ra \Bbase$ be a projective and smooth family of geometrically connected curves.
We record such a family of curves as
\begin{equation}
\label{smcurve}
    C/\Bbase/J.
\end{equation}

{\bf Rank $r$ and degree $d.$}
We fix the rank $r$ and degree $d$ of the vector bundles underlying Higgs bundles, connections and $t$-connections. When relevant, in context, we make further  assumptions on rank and degree, and sometimes on the characteristic.

{\bf The Hodge Moduli Space.}
A $t$-connection on $C/B$ is a triple $(E, t,\nabla_t)$, where $E$ is vector bundle, $t\in H^0(\Bbase,\mathcal{O}_{\Bbase})$, and $\nabla_t:E\ra E\otimes_{\mathcal{O}_{\Bbase}}\omega_{C/\Bbase}$ 
is an $\mathcal{O}_{\Bbase}$-linear morphism of $\mathcal{O}_C$-modules so that for every $f$, a local section of $\mathcal{O}_E$, and $s$, a local section of $E$, we have that $\nabla_t(fs)=tdf\otimes s+f\nabla_t(s)$.

By \cite[Thm.\;1.1]{la-2014}, there exists a quasi projective $\Bbase$-scheme $M_{Hod}(C/\Bbase)$,
which is the coarse moduli space of slope semistable $t$-connections of rank $r$ and degree $d$ on $C/\Bbase$.

\begin{rmk}\la{jj}
For the notions of universally/uniformly corepresenting, see \cite[Thm.\;1.1]{la-2014}.
The coarse moduli space uniformly corepresents  the functor of semistable families; the stable part is open and universally corepresents the functor of stable families; when rank and degree are coprime, stability equals semistability, and we have universal corepresentability. In particular, in the stable case, taking fibers commutes with
taking the coarse moduli space. 
\end{rmk}

Considering $t$ as a section of $\mathbb A^1_{\Bbase}$ over $\Bbase$, the assignment $(t,E,\nabla_t)\mapsto t$ defines a natural morphism of $\Bbase$-schemes:
\begin{equation}
\label{deftau}
\tau_{Hod}(C/\Bbase): M_{Hod}(C/\Bbase)\longrightarrow \mathbb A^1_{\Bbase}.
\end{equation}

{\bf Frobenius.}
Let  $J$ be of characteristic $p>0$, with $p$ a prime number. 
Let $q: T\ra \Bbase$ be a $\Bbase$-scheme.
Let $fr_T: T\ra T$ be the absolute Frobenius, i.e.  the identity on the topological space,  with comorphism $a \mapsto a^p.$
Let $T^{(\Bbase)}:=T\times_{\Bbase, fr_{\Bbase}} \Bbase$ be the Frobenius twist of $T$ relative to $B$.
We have the following commutative diagram
\begin{equation}
    \xymatrix{
    T\ar[rr]_-{Fr_T} \ar@/^1pc/[rrr]^-{fr_T} \ar[drr]_{q} && T^{(\Bbase)} \ar[r]_{\sig_{T}} \ar[d]_-{q^{(\Bbase)}} & T\ar[d]\\
    && B\ar[r]_-{fr_B}& B.
    }
\end{equation}

{\bf The Hodge-Hitchin Morphism.}
Let $J$ be a field of characteristic $p>0$.
Given any $t$-connection $\nabla_t$ on $C/\Bbase$,
\cite[\S3.5]{la-pa} defines the $p$-curvature $\Psi(\nabla_t)$ of $\nabla_t$, 
which is an $\mathcal{O}_C$-linear morphism $E\ra E\otimes_{\mathcal{O}_{\Bbase}}\omega_{C/\Bbase}^{\otimes p}$.
Let $A(C/B,\omega_{C/\Bbase}^p)$ be the vector bundle associated with the locally free sheaf 
$\bigoplus_{i=1}^r\pi_*\omega_{X/\Bbase}^{\otimes ip}$  (recall that we have fixed rank $r$ and degree $d$  for the Hodge moduli space).
Taking the characteristic polynomial of $\Psi(\nabla_t)$ defines a morphism 
$cp: M_{Hod}(C/B)\ra A(C/\Bbase, \omega_{C/\Bbase}^p)$.
Let $A(C^{(\Bbase)}/\Bbase, \omega_{X^{(\Bbase)}/\Bbase})$
be the total space of the vector bundle $\bigoplus \pi^{(B)}_*\omega_{X^{(\Bbase)}/\Bbase}^{\otimes i}$.
The Frobenius pull back $Fr_{C/\Bbase}^*$ defines a closed immersion 
$A(C^{(\Bbase)}/\Bbase,  \omega_{X^{(\Bbase)}/\Bbase})\hra A(C/\Bbase, \omega_{C/\Bbase}^p)$.
\cite[Prop.\;3.2]{la-pa} shows that 
there exists natural factorization of $(cp, \tau_{Hod}): M_{Hod}(C/\Bbase)\ra A(C/\Bbase, \omega_{C/\Bbase}^p)\times \mathbb A^1_{\Bbase}$ as
\begin{equation}
\label{eqdefh}
    M_{Hod}(C/\Bbase) \xra{h_{Hod}(C/B)} A(C^{(\Bbase)}/\Bbase, \omega_{X^{(\Bbase)}/\Bbase})\times\mathbb A^1_{\Bbase}
    \xra{(Fr_{C/\Bbase}^*, \text{Id}_{\mathbb A^1_{\Bbase}})} A(C/\Bbase, \omega_{C/\Bbase}^p)\times \mathbb A^1_{\Bbase}.
\end{equation}

The quasi-projective morphism $h_{Hod}(C/\Bbase)$ in (\ref{eqdefh}) is called the Hodge-Hitchin morphism.
Note that \ci[Prop. 3.2]{la-pa} contains a minor inaccuracy, as it declares the target of $H$ to be
$A\times_B \mathbb A^1$.

The diagram (\ref{eqdefh}) is made of $\Bbase$-schemes endowed with $\GB$-actions so that the morphisms are $\GB$-equivariant.
The action on $M_{Hod}$ is given by $t\cdot (E, \nabla_{s}):= (E, \nabla_{ts})$. 
Let $\mathbb A_i'$ (resp. $\mathbb A_i^p$) be the direct factor of $A(C^{(\Bbase)}/\Bbase)$ 
(resp. $A(C/\Bbase,\omega_{C/\Bbase}^p)$) that is the vector bundle associated to the locally free sheaf $\pi_*^{(\Bbase)}\oxb^{\otimes i}$ (resp. $\pi_*\omega_{X/\Bbase}^{ip}$).
The action on $A(C^{(\Bbase)}/\Bbase)$ (resp. $A(C/\Bbase,\omega_{C/\Bbase}^p)$) is given by the standard dilation weight $ip$ actions
on each  direct factor $\mathbb A'_i$ (resp. $\mathbb A_i^p$). 
The action on $\mathbb A^1_{\Bbase}$ is the usual weight one dilation action.

{\bf Dolbeault Moduli Space and Hitchin Morphism.}
Let $M_{Dol}(C/B)$ be the coarse moduli space of slope semistable Higgs bundles of fixed rank $r$ and degree $d$ on $C/\Bbase.$
$M_{Dol}(C/\Bbase)$ is quasi-projective, see \cite[Thm.\;1.1]{la-2014}. 
Let $A(C/\Bbase; \omega_{C/\Bbase})$ be the vector bundle associated to the locally free sheaf
$\bigoplus_{i=1}^r \pi_*\omega_{C/\Bbase}^{\otimes i}$.
Let 
\beq\la{hitz mo}
h_{Dol}(C): M_{Hod}(C/\Bbase)\longrightarrow A(C/\Bbase,\omega_{C/\Bbase})
\eeq
be the Hitchin morphism that sends a Higgs field to the coefficients of its characteristic polynomial.

If $J$ is a field of positive characteristic, then
there exists a natural isomorphism $A(C/\Bbase,\omega_{C/\Bbase})^{(\Bbase)}\cong A(C^{(\Bbase)}/\Bbase,\omega_{C^{(\Bbase)}/\Bbase})$ (See Lemma \ref{lemma frfr}).
Let 
$h_{Hod}(C/B)_{0_{\Bbase}}: M_{Hod}(C/\Bbase)_{0_{\Bbase}}\ra A(C^{(\Bbase)}/\Bbase)$
be the base change of the Hitchin morphism via the closed immersion $0_{\Bbase}\hra \mathbb A^1_{\Bbase}$.
There exists a natural morphism $M_{Dol}(C/\Bbase)\ra M_{Hod}(C/\Bbase)_{0_{\Bbase}}$ that is bijective on geometric points.
Lemma \ref{frfactor} shows that there exists the following commutative diagram of $\GB$-equivariant morphisms
\begin{equation}
\label{hfactor}
\xymatrix{
    M_{Hod}(C/\Bbase)_{0_{\Bbase}} \ar[r]^-{h_{Hod,0_{\Bbase}}} &
    A(C^{(\Bbase)}/\Bbase) \ar[r]^-{\simeq} &
    A(C/\Bbase)^{(\Bbase)}\\
    M_{Dol}(C/\Bbase) \ar[r]_-{h_{Dol}} \ar[u]&
    A(C/\Bbase). \ar[ur]_-{\;\;\;\;Fr_{A(C/\Bbase)/\Bbase}}
    }
\end{equation}

{\bf de Rham Moduli Space and de Rham-Hitchin Morphism.}
A flat connection is a $t$-connection with $t=1$.
Let $M_{dR}(C)$ be the moduli space of semistable flat connections of fixed rank $r$ and degree $d$.
By \cite[Thm.\;1.1]{la-2014}, the de Rham moduli space $M_{dR}(C)$ is quasi-projective.

When $J$ is a field of positive characteristic, 
there is the natural morphism $M_{dR}(C)\ra M_{Hod}(C)\times_{\mathbb A^1_{\Bbase}} 1_{\Bbase}$, 
which, by Lemma \ref{dreq}, is an isomorphism.
The restriction $h_{dR}(C)$ of $h_{Dol}(C)$ to $M_{dR}(C)$ is called the de Rham-Hitchin morphism. 
Lemma \ref{dreq}.(2) shows that the restriction 
$h_{Hod}(C)_{\GB}: M_{Hod}(C)\times_{\mathbb A^1_{\Bbase}} \GB\ra A(C^{(\Bbase)})\times_{\Bbase}\GB$ admits a $\GB$-equivariant trivialization as $h_{dR}(C)\times \text{Id}: M_{dR}(C)\times_{\Bbase} \GB\ra A(C^{(\Bbase)})\times_{\Bbase} \GB$.


{\bf Statement of Results.}

Our first result is the Projective Completion of $\tau: M_{Hod} \to \mathbb A^1$ Theorem \ref{comphodj},
to the effect that there is a natural $\Gm$-equivariant projective completion  $\ov{\tau}: \ov{M_{Hod}} \to \mathbb A^1$ of the morphism $\tau: M_{Hod}
\to \mathbb A^1$.  
If we further require that the base ring $J$ is a field of positive characteristic, 
we can also extend the Hodge-Hitchin morphism $h_{Hod}: M_{Hod} \to A(C^{(\Bbase)}/\Bbase)\times \mathbb A^1$ and prove that
the Hodge-Hitchin morphism $h_{Hod}$ is proper, in fact projective (Theorem \ref{cpt tm hod}). 
To our knowledge, the properness of $h_{Hod}$ has not been addressed before.
\ci[Prop. 5.1]{la-pa} addresses the special case of  nilpotent $t$-connections, i.e. the properness of $h_{Hod}$ over the locus $0_{A} \times \mathbb A^1$. 
In our proof, we leverage on  this special case;
in fact, we only need the properness of the nilpotent cone $N_{Dol}$, i.e. that of the Hitchin fiber 
$h_{Dol}^{-1}(0_{A})$ over the origin 
$0_{A}\in A(C/\Bbase)$.
The special case  of this projectivity result when rank and degree are coprime (which rules out the case of degree zero, for example) is established
by an ad hoc method in \ci{de-szh naht}.

 \begin{tm}[\bf Projective Completion of $\tau: M_{Hod} \to \mathbb A^1$]
 \label{comphodj}
Let the smooth curve $C/\Bbase/J$ be as in (\ref{smcurve}).
We have the following commutative $\GB$-equivariant diagram
\begin{equation}
\label{eq001}
    \xymatrix{
    M_{Hod}(C/\Bbase) \ar@{^{(}->}[r] \ar[d]_{\tau_{Hod}(C/\Bbase)} & 
    \ov{M_{Hod}(C/\Bbase)} \ar[d]^-{\ov{\tau_{Hod}(C/\Bbase)}}\\
    \mathbb A^1_{\Bbase}\ar[r]^-{\simeq} & \mathbb A^1_{\Bbase}, 
    }
\end{equation}
where :
\ben
\item
The top horizontal arrow is an open immersion with dense image, dense in every fiber of $\ov{\tau_{Hod}(C/\Bbase)}$;
\item
The morphism $\ov{\tau_{Hod}(C/\Bbase)}$ is projective.
\een
\end{tm}

\begin{tm}[\bf Projective Completion of the Hodge-Hitchin Morphism]\la{cpt tm hod} $\;$
In the setup in Theorem \ref{comphodj}, if we further assume that $J$ is a field of characteristic $p>0$,
then
we have  the following commutative  $\GB$-equivariant diagram

\beq\la{eq01}
\xymatrix{
M_{Hod}(C/\Bbase)  \ar@{^{(}->}[rr] \ar[d]^-{h_{Hod}(C/\Bbase)}    \ar@/_4pc/[dd]_-{\tau_{Hod}(C/\Bbase)}& & \ov{M_{Hod}(C/\Bbase)}  \ar[d]_-{\ov{h_{Hod}(C/\Bbase)}} \ar@/^4pc/[dd]^-{\ov{\tau_{Hod}(C/\Bbase)}}
\\
A(C^{(\Bbase)}/\Bbase) \times \mathbb A^1_{\Bbase}  \ar@{^{(}->}[rr]  \ar[d]_-{pr} & & \ov{{A(C^{(\Bbase)}/\Bbase)}} \times \mathbb A^1_{\Bbase} \ar[d]^-{\ov{pr}}
\\
 \mathbb A^1_{\Bbase} \ar[rr]^-\simeq && \mathbb A^1_{\Bbase},
}
\eeq
where:

\ben
\item
The top square is Cartesian,  the  horizontal arrows are open immersions with dense  image,  dense in every fiber of 
$\ov{\tau_{Hod}(C/\Bbase)}$ and $\ov{h_{Hod}(C/\Bbase)}$.

\item
The morphisms $h_{Hod}(C/\Bbase), \ov{h_{Hod}(C/\Bbase)}$ and $\ov{pr}$ are proper, in fact projective ($pr$ is affine).

\item 
$\ov{A(C^{(\Bbase)}/\Bbase)}$ is the weighted projective space
${\mathbb P}(1, 1\cdot p, 2p, \ldots, rp)={\mathbb P}(1, 1, 2, \ldots, r)$ associated with the $\GB$-variety
$\mathbb A^1\times  \prod_{i=1}^r \mathbb A'_{i}$, where $\Gm$ acts as standard dilations
of  weight $1$ on  $\mathbb A^1$  and of weight $ip$ on the remaining factors.

\een
\end{tm}
The proofs of Theorems \ref{comphodj} and \ref{cpt tm hod} are postponed to \S\ref{bnbn}.

\begin{rmk}\la{wz}
For the stated  equality of weighted projective spaces, i.e.  keep the first $1$ and replace  $ip$ by $i$ for $i=1,\ldots, r$, see \ci[Prop. 1.3]{delorme}
and \ci[\S1.3, Proposition]{dol}. This should not be confused with the  fact that, when dealing with weighted projective spaces,
we can replace the vector of weights by a positive  integer multiple of it.
\end{rmk}

\begin{rmk}\la{rmk ala}
A. Langer 's \ci[statement at the top of p. 531 and Th. 5.1]{la-2014} implies  that $h_{Hod}$ is proper.
On the other hand, in order to have a complete proof of loc.cit.,  one  also needs to prove that the morphism
from the moduli of semistable bundles with $t$-connections to the appropriate moduli space of semistable
Higgs bundles (with Higgs field then given by the $p$-curvature of the $t$-connection), is proper.
A. Langer has very kindly provided us with a proof of this fact in a private communication.
Added in revision: the paper \ci{La-2021} provides complete details and proves an even stronger statement.
\end{rmk}

By taking the fiber over $1_{\Bbase} \in \mathbb A^1_{\Bbase}$ of (\ref{eq001}) and (\ref{eq01}), and by observing that the fiber of $\ov{\tau}$ over the same value is Simpson's compactification $\ov{M_{dR}}$, we immediately deduce the following Theorems \ref{Jcpt tm dr} and \ref{cpt tm dr}:

\begin{tm}[\bf Projective Completion of $M_{dR}$]
\label{Jcpt tm dr}
Let the smooth curve $C/\Bbase/J$ be as in (\ref{smcurve}).
There exists a projective $\Bbase$-scheme $\ov{M_{dR}(C/\Bbase)}$ and an open immersion of $\Bbase$-schemes
$M_{dR}(C/\Bbase)\hra \ov{M_{dR}(C/\Bbase)}$ with dense image.
\end{tm}

\begin{tm}[\bf Projectivity of the de Rham-Hitchin morphism]\la{cpt tm dr} $\;$
In the setup in Theorem \ref{Jcpt tm dr}, if we further require that $J$ is a field of characteristic $p>0$, then
we have the following Cartesian diagram
\beq\la{eq4}
\xymatrix{
M_{dR}(C/\Bbase)\ar[d]_-{h_{dR}=h_{Hod,1_{\mathbb A^1}}}  \ar@{^{(}->}[r]  
& \ov{M_{dR}(C/\Bbase)} \ar[d]^-{\ov{h_{Hod,1_{\mathbb A^1}}}}
\\
{A(C^{(\Bbase)}/\Bbase})  \ar@{^{(}->}[r]  & \ov{{A(C^{(\Bbase)}/\Bbase)}},
}
\eeq
where 
\ben
\item
The horizontal arrows are open embeddings with dense image;
\item
The morphisms $h_{dR}=h_{Hod,1_{\mathbb A^1}}$ and $\ov{h_{Hod,1_{\mathbb A^1}}}$ are projective;
\item
$\ov{A(C^{(B)}/\Bbase)}$ is the weighted projective space in Theorem \ref{cpt tm hod}.(3);

\item
The compactification  $\ov{M_{dR}}$ is projective.
\een
\end{tm}

In the Dolbeault case, we do not know whether the natural $\GB$-equivariant morphism $M_{Dol} \to M_{Hod,0_{\mathbb A^1}}$
is an isomorphism. 
On the other-hand, we use the $\GB$-action to obtain  projective $\GB$-equivariant completions of $M_{Dol}$
and of $M_{Hod,0_{\mathbb A^1}}$
which are suitably compatible with the natural $\GB$-equivariant morphism $M_{Dol} \to M_{Hod,0_{\mathbb A^1}}$. Note that we use the subscripts to denote the fibers: for example, the fiber of $h_{Hod}$ over $0_{\mathbb{A}^1}$ is denoted by $M_{Hod, 0_{\mathbb{A}^1}}$.

\begin{tm}[\bf Projective Completion of the Dolbeault Moduli Space]
\label{Jcpt tm dol}
Let the smooth curve $C/\Bbase/J$ be as in (\ref{smcurve}).
We have the following commutative $\GB$-equivariant diagram
\begin{equation}
\label{ahmm}
    \xymatrix{
    A(C/\Bbase)\ar[d]&&
    M_{Dol}(C/\Bbase) \ar[r] \ar[d] \ar[ll]_-{h_{Dol}}&
    M_{Hod,0_{\mathbb A^1}}(C/\Bbase) \ar[d]\\
    \ov{A(C/\Bbase)}&&
    \ov{M_{Dol}(C/\Bbase)} \ar[r] \ar[ll]_{\ov{h_{Dol}}}&
    \ov{M_{Hod,0_{\mathbb A^1}}(C/\Bbase)},
    }
\end{equation}
where:
\ben
\item
All the $\Bbase$-schemes in the bottom row of (\ref{ahmm}) are projective;
\item
All the vertical arrows are open immersions with dense image;
\item
All the horizontal arrows in (\ref{ahmm}) are projective;
\item
The two horizontal arrows in the right half of (\ref{ahmm}) are bijective on geometric points and, if degree and rank are coprime, then they are isomorphisms.
\item
The morphism $\ov{h_{Dol}}:\ov{M_{Dol}} \to \ov{A}$ is naturally isomorphic to  the compactification constructed in \ci[Thm.\;3.1.1, which uses Thms.\;3.2.1 and 3.2.2]{de cpt} (loc.\;cit.\;works over the complex numbers, but in view of the compactification and projectivity results of this paper, the construction
and results hold  in arbitrary characteristic as well).
\een
\end{tm}

When the base ring $J$ is a field of positive characterisic, the Hodge-Hitchin morphism exists, and we can slightly improve Theorem \ref{Jcpt tm dol} as following:

\begin{tm}[\bf Projectivity of the  Hitchin morphism]\la{cpt tm dol} $\;$
In the setup in Theorem \ref{Jcpt tm dol}, if we further require that $J$ is a field of characteristic $p>0$, then
we have the following $\GB$-equivariant commutative diagram:
\begin{equation}
    \label{eq02}
    \xymatrix{
    \ov{M_{Dol}(C/\Bbase)} \ar[rrr] \ar[ddd]_-{\ov{h_{Dol}}}&&&
    \ov{M_{Hod,0_{\mathbb A^1}}(C/\Bbase)} \ar[ddd]^-{\ov{h_{Hod,0_{\mathbb A^1}}}}\\
    & M_{Dol}(C/\Bbase) \ar[r] \ar[d]_-{h_{Dol}} \ar[ul] & M_{Hod,0_{\mathbb A^1}}(C/\Bbase) \ar[ur]\ar[d]^-{h_{Hod,0_{\mathbb A^1}}}&
\\
&A(C/\Bbase) \ar[dl] \ar[r]^-{Fr_A} & {A(C^{(\Bbase)}/\Bbase}) \ar[dr]\\
\ov{A(C/\Bbase)} \ar[rrr]^-{Fr_{\ov{A}}} &&& \ov{{A(C^{(\Bbase)}/\Bbase)}},
    }
\end{equation}

where
\ben
\item
All the oblique arrows in (\ref{eq02}) are open immersions with dense image;
\item
All vertical and horizontal arrows in (\ref{eq02}) are projective morphisms;
\item
The top two horizontal arrows satisfy the properties in Theorem \ref{Jcpt tm dol}.(4);
\item
If we further require that $J$ is algebraically closed, then the top two arrows are universal homeomorphisms;
\item 
$\ov{A}$ is the   weighted projective space ${\mathbb P}(1,1,2,\ldots, r)$ associated with the $\GB$-variety
$\mathbb A^1\times  \prod_{i=1}^r \mathbb A'_{i}$, where $\GB$ acts as standard dilations
of  weight $1$ on  $\mathbb A^1$  and of weight $i$ on the remaining factors and, as the notation indicates, the morphism 
$Fr_{\ov{A}}$ is the relative Frobenius morphism for the $\Bbase$-scheme $\ov{A}.$
\een
\end{tm}
The proofs of  Theorems \ref{Jcpt tm dol} and \ref{cpt tm dol} are postponed to \S\ref{qnbn}.

\begin{rmk}\la{rmk fe-ma} 
We have borrowed the construction of $\ov{\tau}$ in (\ref{eq01})  from 
\ci[Thm. 3.2]{fe-ma}, where it is proved, over the complex numbers, that Simpson's compactification of the Dolbeault moduli space is projective. 
\end{rmk}

\section{Proof of the Compactification Theorem I \ref{ugm proper}}\label{proof ugm proper}


The purpose of this section is to prove the Compactification Theorem I \ref{ugm proper}.  
In \S\ref{prep lemmata}, we prove some well known lemmata that are used in later sections.
In \S\ref{section on Z}, we construct an object $Z(r)$ that is used in the proofs of all the items of Theorem \ref{ugm proper}.
In \S\ref{closed Y}, we prove item (1) of Theorem \ref{ugm proper} which states that $Y_+$ and $Y_-$ are closed inside $X$ and that the uniform geometric quotient $U/\GB$ exists.
In \S\ref{quotient ex}, we prove the item (2) of Theorem \ref{ugm proper} which states that $U/\GB$ exists as a uniform geometric quotient.
In \S\ref{proof of uni cl}, we prove the item (3) of Theorem \ref{ugm proper} which states that $U/\GB\ra S$ is universally closed.
In \S\ref{proof of sep}, we prove the item (4) of Theorem \ref{ugm proper} which states that $U/\GB\ra S$ is separated.

\subsection{Some Preparatory Lemmata}\label{prep lemmata}$\;$

In this subsection, we prove Lemmata \ref{valuative uni cl}, \ref{group on blow up}, and \ref{group on normalization}. 
While they are all well-known and quite general,  we could not find formal 
references for  the exact statements  that we need in the subsequent sections.

We use the following version of valuative criterion in the proof of universal closedness, which follows from \cite[Ex.II.4.11.(b)]{hartshorne} and the proof of \cite[\href{https://stacks.math.columbia.edu/tag/03K8}{03K8}]{stacks}:

\begin{lemma}[\bf Valuative Criterion for Universal Closedness]\label{valuative uni cl}
Let $f:X\ra S$ be a morphism of finite type between noetherian schemes.
Then $f$ is universally closed iff given any DVR $R_0$ inside its fraction field $K$ and a commutative diagram
\begin{equation}
\xymatrix{
\Spec(K) \ar[r] \ar[d] & X \ar[d]^-{f} \\
\Spec(R_0) \ar[r] & Y,
}
\end{equation}
we can find a field extension $L\supset K$, a valuation ring $R$ inside $L$ dominating $R_0$, and a morphism $\Spec(R)\ra X$, making the following diagram commutative:
\begin{equation}
\xymatrix{
\Spec(L) \ar[r] \ar[d] & \Spec(K) \ar[r] \ar[d] & X \ar[d]^-{f}\\
\Spec(R) \ar[r]\ar[rru] & \Spec(R_0) \ar[r] & Y.
}
\end{equation}
\end{lemma}

We need two Lemmata about lifting group actions along normalizations and blowing ups in the proof of Lemma \ref{eachz}.

\begin{lemma}[\bf Group actions on blowups]
\label{group on blow up}
Let $X$ be a locally noetherian scheme over a scheme $T$. 
Let $G$ be a group scheme over $T$ that is locally noetherian.
Let $\mu: G\times_T X\ra X$ be a $G$-action on $X$.
Let $Y$ be a $G$-invariant closed subscheme of $X$ with dense complement.
Let $\widetilde{X}$ be the blowing up of $X$ along $Y$.
Suppose that $G\times_T G\times_T \widetilde{X}$, $G\times_T \widetilde{X}$, and $\widetilde{X}$ are reduced, and that $X$ is separated.

Then the blowing up $\widetilde{X}$ of $X$ along $Y$ admits a $G$-action making the blow down morphism $\pi: \widetilde{X}\ra X$ to be $G$-equivariant.
\end{lemma}
\begin{proof}
We have the canonical isomorphisms
\begin{equation}
\label{gty}
    G\times_T Y\cong (G\times_T X)\times_{p_X, X} Y\cong (G\times_T X)\times_{\mu, X} Y,
\end{equation}
where $p_X$ is the natural projection $G\times X\ra X$, the morphisms from $Y$ to $X$ are always the inclusion $Y\hra X$, and we have included the morphisms in the subscript to emphasize which fiber product we are taking.
Indeed, the first isomorphism is automatic, and the second isomorphism follows from the $G$-invariance of $Y$.

Since $p_X$ is flat, by \cite[Prop 8.1.12.(c)]{liu}, the blow up of $G\times_T X$ with center $G\times_T Y$ is canonically isomorphic to $G\times_X \widetilde{X}$. By (\ref{gty}), we see that 
$G\times_X\widetilde{X}$ is also the blow up of $G\times_T X$ with center the fiber $\mu^{-1}(Y)$.
By the universal property of the  blow up \cite[Prop 8.1.15]{liu}, there exists a unique morphism $\widetilde{\mu}: G\times_T \widetilde{X}\ra G\times_T X$, making the following diagram commutative:
\begin{equation}
\xymatrix{
    G\times_T \widetilde{X} \ar[r]^-{\widetilde{\mu}} \ar[d] &\widetilde{X} \ar[d]\\
    G\times_T X \ar[r]_-{\mu} & X.
    }
\end{equation}

Let $\mu^G:G\times_T G\ra G$ be the group multiplication morphism and $e: T\ra G$ be the identity morphism.
To verify that $\widetilde{\mu}$ defines a group action on $\widetilde{X}$, and that $\pi:\widetilde{X}\ra X$ is $G$-equivariant, we need to show the following three identities of morphisms:
\begin{equation}
    \widetilde{\mu}\circ (1_G\times \widetilde{\mu})=\widetilde{\mu} \circ (\mu^G \times 1_{\widetilde{X}}): \;G\times_T G\times_T \widetilde{X}\ra \widetilde{X},
\end{equation}
\begin{equation}
    \widetilde{\mu}\circ (e\times 1_{\widetilde{X}})=1_{\widetilde{X}}: \;\widetilde{X}\ra \widetilde{X},
\end{equation}
\begin{equation}
\label{equivariant}
    \mu\circ (1_G\times \pi) =\pi\circ \widetilde{\mu}: \;\;\;G\times_T \widetilde{X} \ra X.
\end{equation}

All three pairs of morphisms agree on the open and dense subscheme corresponding to $X\setminus Y$.
By the assumption on the reducedness and separatedness on the domains and target, \cite[Ex II.4.2]{hartshorne}  shows that the three identities hold over all of their domains. (Note that the separatedness of $X$ implies the separatedness of $\tilde{X}$ by \cite[\href{https://stacks.math.columbia.edu/tag/01O2}{01O2}]{stacks}).
\end{proof}

\begin{lemma}[\bf Group actions on normalization]
\label{group on normalization}
Let $X$ be a scheme over a scheme $T$.
Suppose that $X$ is an integral scheme.
Let $G$ be a group scheme over $T$ that acts on $X$ via the action morphism $\mu: G\times_T X \ra X$.
Let $\pi: X'\ra X$ be the normalization of $X$.
If $G\times_T X'$ is a normal and integral scheme, then $X'$ admits a $G$-action, making the normalization morphism $\pi: X'\ra X$ to be $G$-equivariant.
\end{lemma}
\begin{proof}
Consider the surjective morphism
$G\times_T X' \xra{1_G\times \pi} G\times_T X \xra{\mu} X$.
Since $G\times_T X$ is normal and integral, the universal property of normalization \cite[Prop 12.44]{gowe}  induces a unique morphism 
$\mu': G\times_T X' \ra X'$ so that we have the equality of morphisms:
\begin{equation}
\label{normal 3}
    \mu\circ (1_G\times \pi) =\pi\circ \mu': \;\;\;G\times_T X' \ra X.
\end{equation}

If we can show that $\mu'$ is a $G$-action, then (\ref{normal 3}) shows that the normalization $\pi:X'\ra X$ is $G$-equivariant. We now proceed to show that $\mu'$ is indeed a $G$-action:
Let $\mu^G:G\times_T G\ra G$ be the group multiplication morphism and $e: T\ra G$ be the identity morphism.
We need to show the following two equalities of morphisms:

\begin{equation}
\label{normal 1}
    \mu'\circ (1_G\times \mu')=\mu' \circ (\mu^G \times 1_{X'}): \;G\times_T G\times_T X'\ra X',
\end{equation}
\begin{equation}
\label{normal 2}
    \mu'\circ (e\times 1_{X'})=1_{X'}: \;X'\ra X'.
\end{equation}

To show (\ref{normal 1}), by the uniqueness of $\mu'$, 
it suffices to show that the two morphisms are equal after a composition of $\pi: X'\ra X$, i.e., we want to show that 
\begin{equation}
\label{normal 1'}
    \pi\circ \mu'\circ (1_G\times \mu')=\pi\circ \mu' \circ (\mu^G \times 1_{X'}): \;G\times_T G\times_T X'\ra X' \ra X.
\end{equation}

By (\ref{normal 3}), both the morphisms in (\ref{normal 1'}) factors as
\begin{equation}
   G\times_T G\times_T\times X' \xra{ 1_G\times 1_G\times \pi } G\times_T G\times_T \times X\ra X,
\end{equation}
where the last morphism is 
\begin{equation}
    \mu\circ (1_G\times \mu)=\mu \circ (\mu^G \times 1_{X}): \;G\times_T G\times_T X\ra X.
\end{equation}

We thus have (\ref{normal 1}).
The equality (\ref{normal 2}) follows similarly: by (\ref{normal 3}), we have 
\begin{equation}
    \pi\circ \mu'\circ (e\times 1_{X'}) = \mu\circ \pi\circ (1_G\times \pi) \circ (e\times 1_{X'}) =\mu'\circ (e\times 1_X)=1_X,
\end{equation}
so we have that (\ref{normal 2}) holds after composing $\pi$, hence (\ref{normal 2}) holds by the uniqueness of $f'$ in the universal property of normalization.
\end{proof}

\subsection{The schemes \texorpdfstring{$Z(r)$}{Z(r)}}\label{section on Z}$\;$

Our goal in \S\ref{proof ugm proper} is to prove Theorem \ref{ugm proper}.

A feature of our proof of Theorem \ref{ugm proper} is that the proof of $Y_+$ and $Y_-$ are closed is similar to the proof of the universal closedness of $U/\GB\ra S$, in the sense that both proofs start with a point $r\in X(R)$ for some discrete valuation ring $R$, which then produces a rational map $\Proj^1_R\dashrightarrow X_R$, and both proofs rely heavily on the structure of the resolution of indeterminancy $Z(r)$ of the rational map $\Proj^1_R\dashrightarrow X_R$. Therefore it seems best to first introduce and study $Z(r)$ in this Section \ref{section on Z} and then to diverge to separate proofs of the items (1)-(2) of Theorem \ref{ugm proper} in Sections \ref{closed Y}-\ref{proof of sep}.\\

Our next goal is to define what $Z(r)$ is. We do so  in Lemma \ref{lemma int Z}. 

Let $R$ be a discrete valuation ring with fraction field $L$ and residue field $\kappa$.

Let $r\in X(R)$. 
Let $\eta\in X(L)$ be the restriction of $r$ to the open subscheme $\Spec(L)\subset \Spec(R)$.
Taking the orbit of $\eta$, we have the morphism $\mu_{\eta}: \Gg_{m,L}\ra X$. 
Since $\GB$ acts trivially on $S$, we have that the image of the composition $\Gg_{m,L}\xra{\mu_{\eta}}X\ra S$ is a point, and that the morphism $\Gg_{m,R}\xra{\mu_{\eta}} X\ra S$ factors through a morphism $\Spec(R)\ra S$.
Therefore we can extend the morphism $\Gg_{m,L}\ra S$ to a morphism $\Proj^1_L\ra S$. 
Since $X$ is proper over $S$, we can extend $\mu_{\eta}$ to an $S$-morphism $\overline{\mu_{\eta}}: \Proj^1_L\ra X$. 
We then have a graph morphism 
$\Gamma_{\ov{\mu_{\eta}}}: \Proj^1_L\ra \Proj^1_L\times_S X$, which is a closed immersion \cite[p.106]{hartshorne}, and is $\Gg_{m,L}$-equivariant.
Let $j:\Proj^1_L\times_S X\ra \Proj^1_R\times_S X$ be the natural open immersion induced by the open immersion $\Spec(L)\hra \Spec(R)$.
Let $W$ be the scheme theoretic image of $j\circ \Gamma_{\ov{\mu_{\eta}}}: \Proj^1_L\ra \Proj^1_R\times_S X$.
Since $X$ is projective over $S$, we have that $X_R:= X\times_S \Spec(R)$ is projective over $\Spec(R)$.
The morphism $p_X\circ j\circ \Gamma_{\ov{\mu_{\eta}}}: \Proj^1_L\ra X$ induces a rational map $b: \Proj^1_R \dashrightarrow X_R$.

\begin{lemma}[\bf Introduce $Z(r)$]
\label{lemma int Z}
There exists a proper birational morphism $\pi_5: Z(r)\ra W$ with $Z(r)$ regular, 
making the following diagram commutative:
\begin{equation}
\label{el pt indet}
    \xymatrix{
    \Proj_R^1 \ar[d]_-{\pi_2} \ar@{-->}[dr]^-{b} &
    W\ar[l]_-{\pi_1} \ar[d]^-{\pi_3} & Z(r)\ar[l]_-{\pi_5}\\
    R&
    X_R. \ar[l]^-{\pi_4} &
    }
\end{equation}
We have that $\pi_5$ is an isomorphism above every regular point of $W$. Moreover,
\ben
\item
$Z(r)$ is the last element $W_n$ in the following sequence:
\begin{equation}
\label{W seq}
    \pi_5: Z(r)=W_n\ra W_{n-1} \ra... \ra W_1\ra W_0=W,
\end{equation}
where $W_1\ra W$ is the normalization, and for every $i\ge 1$, $W_{i+1}\ra W_i$ is obtained by first blowing up,
$\w{W_i}\ra W_i$,  the singular locus (which by definition is reduced) of $W_i$, and then  by normalizing, $W_{i+1}\ra \w{W_i}$, the resulting 
$\w{W_i}$;
\item
$Z(r)$ is also the last element $Z_n$ in the following sequence
\begin{equation}
\label{Z seq}
    \pi_6: Z(r)= Z_n \xra{p_n} Z_{n-1} \xra{p_{n-1}} ... \xra{p_1} Z_0= \Proj^1_R,
\end{equation}
where for each $i\ge 1$, $Z_i$ is obtained by blowing up a closed point of $Z_{i-1}$.
\een
\end{lemma}
\begin{proof}

The commutative diagram
(\ref{el pt indet})
is exactly the elimination of points of indeterminacy for the rational map $b$ as constructed in \cite[Thm.\;9.2.7]{liu},
where it is shown that 
every rational map from a regular fibered surface over a one dimensional Dedekind scheme $D$ (such as $\Proj^1_R$ over $\Spec(R)$) to a projective $D$-scheme (such as $X_R$ over $\Spec(R)$) admits an elimination of indeterminacy $\pi_6:Z(r)\ra Z_0$, which is a finite sequence of blowing-ups of closed points of the target as in (\ref{Z seq}), and which factors through the desingularization of the closure of the graph of the rational map as in (\ref{W seq}). We thus obtain our lemma as a direct application of \cite[Thm.\;9.2.7]{liu}.
\end{proof}

Our goal in the remainder of this \S\ref{section on Z} is to prove Lemma \ref{main}, which describes the reduction of the closed fiber of $Z(r)$ over $\Spec(R)$.
We now fix $r\in X(R)$, and suppress the argument $r$ in $Z(r)=Z$.
We start with the following

\begin{lemma}[\bf $\Gg_{m,R}$ action on the partial resolutions $Z_i$]\label{eachz}
Each $Z_i$, $0\le i\le n$, in the sequence (\ref{Z seq})
admits a $\Gg_{m,R}$-action so that each $p_{i}:Z_i\ra Z_{i-1}$ is $\Gg_{m,R}$-equivariant. 
Furthermore, each $p_{i+1}: Z_{i+1}\ra Z_i$ is a blow up of a $\Gg_{m,R}$-fixed closed point of $Z_{i}$. 
\end{lemma}
\begin{proof}

We first show the 

\textbf{CLAIM}: $Z_n=Z$ admits a $\Gg_{m,R}$-action so that $\pi_6: Z\ra \Proj^1_R$ is $\Gg_{m,R}$-equivariant.

Then we finish the proof using an increasing induction on $i\ge 0$.

To show the \textbf{CLAIM} above, by Lemmata \ref{group on blow up}, \ref{group on normalization}, and \ref{lemma int Z}.(1), we see that it suffices to show that for each $i\ge 1$ the singular locus of $W_i$ is made of $\Gg_{m,R}$-fixed points. 
Indeed, since each $W_i$, $i\ge 1$, is normal, the singular locus has to be closed points \cite[Prop.\;4.2.24]{liu}. If these singular closed points are not fixed by the $\Gg_{m,R}$-action, then the orbits of the points under $\Gg_{m,R}$ would form a one dimensional subscheme of the singular locus of $W_i$, contradicting the normality of $W_i$. The \textbf{CLAIM} is thus proved.

We now prove the lemma by induction on $i\ge 0$.

Base case $i=0$:

The $\Gg_{m,R}$-action on $\Proj^1_R$ is the natural one induced by the multiplication on the open subscheme $\Gg_{m,R}\subset \Proj^1_R$. The statement about $p_0$ is vacuous. 
Let $z_0$ be the closed point of $Z_0$ that is the center of the blow up $p_1: Z_1\ra Z_0$. We would like to show that $z_0$ is a $\Gg_{m,R}$-fixed point:

If $z_0$ is not fixed by the $\Gg_{m,R}$-action, then the fiber $\pi_6^{-1}(z_0)$ inside $Z$ has an irreducible component $E$ that is not stable under the $\Gg_{m,R}$-action. 
Let $\xi_E$ be the generic point of $E$.
Consider the orbit morphism $\mu_{\xi_E}:\Gg_{m,k(\xi_E)}\ra Z$.
The scheme theoretic image $O(E)$ of $\mu_{\xi_E}$ properly contains $E$ as a closed subscheme since $E$ is not $\Gm$-stable. Thus dim$(O(E))\ge 2$.
Since $Z$ is two dimensional and integral, the set theoretic image of $\mu_{\xi_E}$ is dense inside $Z$. 
Therefore, the set theoretic image $\pi_6(Z)$ is contained in the closure of the $\Gg_{m,R}$-orbit of $z_0$, which is one dimensional, contradicting that $\pi_6$ is surjective.
Thus $z_0$ has to be a $\Gg_{m,R}$-fixed point, and the base case is established.

Now suppose that we have established the case $i-1$ and would like to show the case $i$.

Since $p_i: Z_i\ra Z_{i-1}$ is a blow up of a $\Gg_{m,R}$-fixed closed point of $Z_{i-1}$.
By Lemma \ref{group on blow up}, we see that $Z_i$ admits a $\Gg_{m,R}$-action so that $p_i$ is $\Gg_{m,R}$-equivariant.
Let $z_i$ be the closed point of $Z_i$ that is the center of the blow up $p_{i+1}:Z_{i+1}\ra Z_i$.
We would like to show that $z_i$ is a $\Gg_{m,R}$-fixed point.
We first show that the morphism $p_n\circ ... \circ p_{i+1}: Z\ra Z_i$ is $\Gg_{m,R}$-equivariant, i.e.,
the analogue of the equation (\ref{equivariant}) holds in our case:
\begin{equation}
\label{Gm equivariant}
    \ah_{Z_i}\circ (1_G\times (p_n\circ ...\circ p_{i+1}))= p_n\circ ...\circ p_{i+1} \circ \ah_Z: \;\;\Gg_{m,R}\times_{R} Z\ra Z_i, 
\end{equation}
where $\ah_Z$ and $\ah_{Z_i}$ denote the $\Gg_{m,R}$-action morphisms on $Z$ and $Z_i$. 
Since $Z_i$ is obtained from $Z_0=\Proj^1_R$ by iterated blowups on closed points, by the inductive hypothesis we have that the projection $p_i\circ ...\circ p_1: Z_i\ra Z_0$ is a $\Gg_{m,R}$-equivariant isomorphism over an open and dense subscheme $U_0$ of $Z_0$. 
From the \textbf{CLAIM}, we see that 
the equation (\ref{Gm equivariant}) holds when restricted to the open and dense subscheme $\Gg_{m,R}\times_{R}\pi_6^{-1}(U_0)$ of $\Gg_{m,R}\times_{R} Z$. 
Therefore, by \cite[Ex II.4.2]{hartshorne}, we see that the equality (\ref{Gm equivariant}) holds, so $p_n\circ ...\circ p_{i+1}$ is $\Gg_{m,R}$-equivariant. 
We can now use the argument in the base case to conclude. 
Namely, if $z_i$ is not fixed by the $\Gg_{m,R}$-action, then $(p_n\circ ...\circ p_{i+1})^{-1}$ would trace out a dense two dimensional subscheme of $Z$ under the $\Gg_{m,R}$-action. 
We then have that $p_n\circ ...\circ p_{i+1}: Z\ra Z_i$ maps $Z$ to the $\Gg_{m,R}$-orbit of $z_i$ in $Z_i$, which is one dimensional, contradicting that $\pi_6: Z\ra Z_0$ is surjective. 
\end{proof}

Lemma \ref{eachz} above essentially contains all the information of $Z$ in Lemma \ref{main}. 
However, in order to keep track of the $\Gg_{m,R}$-actions under each blow up $p_i: Z_i\ra Z_{i-1}$ in the proof of Lemma \ref{main}, we need to study a family of affine charts involved in the sequence of blow ups (\ref{Z seq}). 
The subtlety is that different $\Gg_{m,R}$-fixed points of $Z_i$ have non-isomorphic, although similar, $\Gg_{m,R}$-invariant affine neighborhood,
and that we need to know what happens to the $\Gg_{m,R}$-action when we blow up any of the $\Gg_{m,R}$-fixed points of $Z_i$.

We start with the affine charts for $Z_0=\Proj^1_R$: we have that $\A^1_R$ and $\Proj^1_R\setminus 0_R\cong \A^1_R$ cover $Z_0$, so the two affine charts are isomorphic to $\Spec(R[x])$ and $\Spec(R[x^{-1}])$, where $x$ is an independent variable. 
The variable $x$ has positive weight $w(x)$ under the $\Gg_{m,R}$-action. 
Let $\lam$ be the uniformizing parameter of $R$, by the triviality of the $\Gg_{m,R}$-action on $\Spec(R)$, we have that the weight of $\lam$ under $\Gg_{m,R}$ is $w(\lam)=0$.

The blow up $p_1: Z_1\ra Z_0$ has center $0_{\kappa}$ or $\infty_{\kappa}$. 
Without loss of generality, assume $p_1$ is the blow up of $0_{\kappa}$. (If $p_1$ is the blow up of $\infty_{\kappa}$ then we can exchange $x$ and $x^{-1}$).
By the description of a blow up algebra of a regular algebra in \cite[p.325 bottom]{liu}, we have that $Z_1$ can be covered by three $\Gg_{m,R}$-invariant affine charts isomorphic to $\Spec(R[x,\lam/x])$, $\Spec(R[y])$ and $\Spec(R[x^{-1}])$, where $y$ is a new independent variable.
We have that the weight of $x$ under the $\Gg_{m,R}$-action is still $w(x)$, while the weight of $\lam/x$ under the $\Gg_{m,R}$-action is $w(\lam/x)=w(\lam)-w(x)=-w(x)$. 
Since the exceptional divisor of $p_1$ is defined by $y$ in $\Spec(R[y])$ and $\lam/x$ in $\Spec(R[x,\lam/x])$, we have that the weight of $y$ under the $\Gg_{m,R}$-action is $w(y)=-w(\lam/x)=w(x)$.

Since $w(\lam/x), w(x)\ne 0$, we have that the $\Gg_{m,R}$-fixed point of $\Spec(R[x,\lam/x])$ is defined by the maximal ideal $\lb x,\lam/x\rb$ generated by $x$ and $\frac{\lam}{x}$. 
The blow up of $\Spec(R[x,\lam/x])$ with center $\lb x, \lam/x\rb$ is covered by the $\Gg_{m,R}$-invariant affine charts $\Spec(R[x,\frac{\lam}{x^2}])$ and $\Spec(R[\frac{x^2}{\lam}, \frac{\lam}{x}])$. As for the weight under the $\Gg_{m,R}$-action, we have that $w(\lam/x^2)=-2w(x)$.

Writing out the charts for iterated blow ups of $\Spec(R[x,\lam/x])$ at $\Gm$-fixed points as above, it is easy to see that we have the following 

\begin{lemma}[\bf  Compatibility of the weights at intersection points]
\label{new bl alg}
For each $Z_i$, $0\le i\le n$, in the sequence (\ref{Z seq}), we have that $Z_i$ can be covered by affine charts that are isomorphic to $\Spec(R[x])$ or 
\begin{equation}
\label{the new alg}
    \Spec(R[z_1, z_2]/(z_1^a z_2^b-\lam)),
\end{equation}
for some $a,b\in \Z_{\ge 0}$, with $a+b>0$. 

The weight of $x$ under the $\Gg_{m,R}$-action is nonzero.
Let $w(z_1)$ and $w(z_2)$ be the weights of $z_1$ and $z_2$ under the $\Gg_{m,R}$-action, 
then we have that $w(z_1)$ and $w(z_2)$ are both nonzero, and that $aw(z_1)+bw(z_2)=0$.
In particular, we have that $w(z_1)w(z_2)<0$.
\end{lemma}
\begin{proof}
From the paragraphs above Lemma \ref{new bl alg}, we see that $Z_i$ is covered by affine charts that are isomorphic to $\Spec(R[x])$ and iterated blow ups of $\Spec(R[x,\lam/x])$ at $\Gm$-fixed points.
We would like to show that such blow ups can be covered by the charts of the form as in (\ref{the new alg}), and that the weights satisfies what in the statement of this Lemma \ref{new bl alg}.

We show this by induction on 
the number  of blow ups of $\Spec(R[x,\lam/x])$. 
The base case, where there are no blow ups, is satisfied because we can take $z_1=x$, $z_2=\lam/x$, $a=b=1$, and we have $w(x)+w(\lam/x)=0$. 
For the inductive step, we blow up a $\Gm$-fixed point of the chart in (\ref{the new alg}). 
Since $w(z_1)$, $w(z_2)\ne 0$, the $\Gm$-fixed point has to be the maximal ideal $\lb z_1, z_2\rb$. 
The resulting blow up can be covered by the spectra of $R[z_1,\frac{z_2}{z_1}]$ and $R[\frac{z_1}{z_2},z_2]$, with $z_1^az_2^b=\lam$, i.e., the spectra of
\begin{equation}
    R[z_1',z_2']/((z_1')^{a+b}(z_2')^b-\lam)\text{ and }R[z_1'',z_2'']/((z_1'')^a(z_2'')^{a+b}-\lam).
\end{equation} 
Furthermore, we have that 
\begin{equation}
    w(z_1')=w(z_1),\;\;w(z_2')=w(z_2)-w(z_1),\;\; w(z_1'')=w(z_1)-w(z_2),\;\; w(z_2'')=w(z_2).
\end{equation}
Now it is easy to check that the weights have the desired properties in the statement of the lemma.
\end{proof}

We can now prove the main Lemma \ref{main} as a corollary of Lemma \ref{new bl alg}:

\begin{lemma}[\bf Shape of Closed Fiber of $Z_i$ over $\Spec(R)$]
\label{main}
For each $Z_i$, $0\le i\le n$, in the sequence (\ref{Z seq}), let 
\beq\la{ei}
E_i:= (Z_i\times_R\kappa)_{red}
\eeq
 be the reduction of the fiber of $Z_i\ra \Spec(R)$ over the closed point $\Spec(\kappa)\hra \Spec(R)$. We have that 
for each $0\le i\le n$,
\ben
\item 
$E_i$ is connected;
\item
the irreducible components of $E_i$ are all isomorphic to $\Proj^1_{\kappa}$;
\item
the singular points of $E_i$ are where the irreducible components of $E_i$ meet, and the singular points are all ordinary double points;
\item
all of the irreducible components of $E_i$ admit nontrivial $\Gg_{m,\kappa}$-actions so that every singular point of $E_i$, which lies in only two irreducible components of $E_i$, is the $0$-limit of one component and the $\infty$-limit of the other component;
\item
the relation $\le$ in Definition \ref{def order} gives a linear order on the set of $\Gg_{m,\kappa}$-fixed points of $E_i$.
Furthermore, there are $i+1$ $\Gg_{m,\kappa}$-fixed points, $i-1$ of which are singular, and the two regular $\Gg_{m,\kappa}$-fixed points are the unique maximal and minimal elements with respect to $\le$.
\een
\end{lemma}

\begin{rmk}[\bf Lemma \ref{main} in terms of graphs]\la{fgo0}
For each $0\le i\le n$, define a directed graph $\Gamma(E_i)$ as follows: 
with each irreducible component of $E_i$ we associate an edge,  and with each $\Gg_{m,\kappa}$-fixed point of $E_i$ we associate a vertex;
according to Lemma \ref{main}.(2),(4), each irreducible component $\mathcal{E}$ of $E_i$ contains two $\Gg_{m,\kappa}$-fixed points $e_1$ and $e_2$, which are the 0 and $\infty$ limits of the $\Gg_{m,\kappa}$-action on $\mathcal{E}$ respectively, thus we can let the vertices corresponding to $e_1$ and $e_2$ be the source and the end of the edge $ed(\mathcal{E})$ corresponding to $\mathcal{E}$ respectively, and let the direction on $ed(\mathcal{E})$ be pointing from the source to the end.
Then Lemma \ref{main} implies that $\Gamma(E_i)$ has $i+1$ vertices, and is of the form
\[\circ \ra \circ \ra ... \ra \circ.\]
\end{rmk}

\begin{proof}[Proof of Lemma \ref{main}]
Since $p_i: Z_i\ra Z_{i-1}$ is the blow up of a closed point of $Z_{i-1}$, we have that $E_i=(p_i^{-1}(E_{i-1}))_{red}$ is connected if and only if $E_{i-1}$ is connected. Since $E_0=\Proj^1_{\kappa}$ is connected, we have that each $E_i$ is connected, thus we have item (1).

From Lemma \ref{new bl alg}, we see that the fiber $Z_i\times_R\kappa$ can be covered by the affine charts that are isomorphic to the spetra of $\kappa[x]$ or $\kappa[z_1,z_2]/(z_1^az_2^b)$ for some $a,b\in \Z_{\ge 0}$ with $a+b>0$.
Therefore $E_i$ can be covered by the affine charts that are isomorphic to the spectra of $\kappa[x]$ or $\kappa[z_1,z_2]/(z_1z_2)$. Items (2) and (3) are immediate from these charts.

From the description of weights in Lemma \ref{new bl alg}, we see that $x$ has nontrivial weights, and that $w(z_1)w(z_2)<0$. Therefore we have item (4).

We prove item (5) by an induction on $i\ge 0$. The base case $i=0$ is automatic because we have that two regular $\Gg_{m,\kappa}$-fixed points $0_{\kappa}$ and $\infty_{\kappa}$ of $E_0=\Proj^1_{\kappa}$, and that $0_{\kappa}\le \infty_{\kappa}$. 
Now suppose we have proved the case $i-1$. Let us order the $i$ $\Gg_{m,\kappa}$-fixed points of $E_{i-1}$ as $z_1\le ...\le z_i$. 
Suppose $p_i: Z_i\ra Z_{i-1}$ is the blow up with center $z_j$ for some $1\le j\le i$. 
We then have that $p_i$ is a $\Gg_{m,R}$-equivariant isomorphism when restricted to the affine charts for $z_1,...,z_{j-1}, z_{j+1},...,z_i$ selected in Lemma \ref{new bl alg}.
For each $k\ne j$, let $y_k:=p^{-1}(z_k)$.
Let $y_j^0$ and $y^{\infty}_j$ be the $0$ and $\infty$ point of the exceptional divisor, which is isomorphic to $\Proj^1_{\kappa}$ by item (2), of the blow up $p_i$. 
It then follows from item (4) that we can linearly order the $\Gm$-fixed points of $E_i$ as 
$y_1\le ...\le y_{j-1}\le y_j^0\le y_j^{\infty}\le y_{j+1}\le...\le y_i$. We have thus showed the first sentence of item (5) in case $i$.
Since each of $y_1$ and $y_i$ lies in only one component of $E_i$, from Lemma \ref{new bl alg} we see that $y_1$ and $y_i$ have affine neighborhoods isomorphic to $\Spec(R[x])$, thus $y_1$ and $y_i$ are regular. 
Any $\Gg_{m,\kappa}$-fixed point of $E_i$ that is not $y_1$ or $y_i$ lies in two components of $E_i$, thus cannot be regular. We have thus showed the case $i$ of item (5).
\end{proof}

\subsection{Closedness of \texorpdfstring{$Y_+$}{Y+} and \texorpdfstring{$Y_-$}{Y-}}
\label{closed Y}$\;$

In this Section \ref{closed Y}, we prove Theorem \ref{ugm proper}.(1) which states that both $Y_-$ and $Y_+$ are closed inside $X$.

The sets $Y_+$ and $Y_-$ are constructible (see the first paragraph of Section \ref{section on Z}), thus it suffices to show that $Y_+$ and $Y_-$ are closed under specializations. 
Let $x, x'$ be two Zariksi points of $X$ so that $x'\in \overline{\{x\}}$. Assume that $x\in Y_+$. We would like to show that $x'\in Y_+$.

By \cite[Lemma II.4.4, Ex.\;II.4.11]{hartshorne}, there exists a discrete valuation ring $R$ with fraction field $L$ and residue field $\kappa$, and an $R$-point $r\in X(R)$, so that $r$ maps the generic point of $\Spec(R)$ to $x$ and the closed point of $\Spec(R)$ to $x'$. 

Take $Z=Z(r)$ as defined in Lemma \ref{lemma int Z}. 
By Lemma \ref{eachz} we have that the morphism $\pi_6: Z\ra \Proj^1_R$ is the composition of iterated $\Gg_{m,R}$-equivariant blow ups at $\Gg_{m,R}$-fixed closed point. 
Therefore $\pi_6$ restricted over the closed subscheme $1_R\subset \Proj^1_R$ is an isomorphism. 
Thus $\pi_6^{-1}(1_R)$ is an $R$-point of $Z$.
We employ the  notation as  in (\ref{el pt indet}).
Define the composition of $\Gg_{m,R}$-equivariant morphisms $\pi_7: Z\xra{\pi_5} W \xra{\pi_3} X_R \xra{p_X} X$.
We then have that $\pi_7|_{\pi_6^{-1}(1_R)}=r: \Spec(R)\ra X$. 
We then have an equality of $L$-points of $X$:
\begin{equation}
\label{pi7pi6}
    \pi_7(\pi_6^{-1}(\infty_L))= \lim_{t\ra \infty} t\cdot x.
\end{equation}

We employ  the notation in Lemma \ref{main}.
By keeping track of what happens to the affine charts containing $\infty_L$ in $\Proj^1_L$ under each blow up $p_i:Z_i\ra Z_{i-1}$ as in Lemma \ref{new bl alg}, we see that $\pi_6^{-1}(\infty_L)$ in $Z$ specializes to the maximal element $z_{n+1}$ with respect to the linear order $\le$. 
By (\ref{pi7pi6}), we have that $\pi_7(z_{n+1})\in \overline{\{\lim_{t\ra \infty} t\cdot x\}}$.
Since $x\in Y_+$, and $V_+$ is closed, we have that $\pi_7(z_{n+1})\in V_+$.
Since $\pi_7$ is $\GB$-equivariant, for any $0\le j\le n+1$, we have that $\pi_7(z_j)\le \pi_7(z_{n+1})$.
By the defining property of $V_+$, we see that $\pi_7(z_j)=\in V_+$.
In particular, we can let $z_j$ be the $\infty$-limit of $\pi_6^{-1}(1_{\kappa})$, we then have that 
\begin{equation}
    \lim_{t\ra \infty} t\cdot x'=\lim_{t\ra \infty} t\cdot r|_{\Spec(\kappa)} =\lim_{t\ra \infty} t\cdot \pi_7|_{\pi_6^{-1}(1_{\kappa})}=\pi_7(z_j)\in V_+.
\end{equation}
Therefore we have that $x'\in Y_+$, and we have proved that $Y_+$ is closed inside $X$.

The proof that $Y_-$ is closed inside $X$ is very similar to the proof above, except that we exchange $0$ and $\infty$ in the argument. We have thus proved Theorem \ref{ugm proper}.(1).

\subsection{Existence of Uniform Geometric Quotient \texorpdfstring{$U/\GB$}{U/GB}}\label{quotient ex}$\;$\\

In this section, we prove Theorem \ref{ugm proper}.(2) which states that a uniform geometric quotient $\phi: U\ra U/\GB$ exists.

\begin{proof}[Proof of Theorem \ref{ugm proper}.(2)]
By Assumption \ref{assumption p}, we can cover $U$ with $\GB$-invariant open affine subschemes $U=\bigcup_i Q_i$.
By \cite[Thm.\;3, Rmk.\;8,\;10]{seshadri}, for each $i$, there exists a uniform categorical quotient $\phi_i:Q_i\ra Q_i/\GB$ which is surjective, and that $Q_i/\GB$ is of finite type over $\Bbase$.

We first show that these $\phi_i$'s glue to form a uniform categorical quotient $\phi:U\ra U/\GB$, which is immediately reduced to showing that for every $i,j$, we have that $Q_i\cap Q_j=\phi^{-1}\phi_i(Q_i\cap Q_j)$, see the proof of \cite[Prop.\;7.9]{alper}.
Since any two closed point $q^i_1, q^i_2$ in $Q_i$ are mapped to the same point in $Q_i/\GB$
if and only if the closures (inside $Q_i$) of the orbits of $q_1^i$ and $q_2^i$ intersect nontrivially, see \cite[Thm.\;3.(ii)]{seshadri}, we are
further reduced to showing that every closed point $q^i$ in $Q_i$ has closed orbit in $Q_i$:

Since the closure of the orbit of $q^i$ in $X$ has boundaries contained in $V$, we have that the orbit of $q^i$ is closed inside $X-V$. By Theorem \ref{ugm proper}.(1), we have that $U$ is open inside $X$, thus the orbit of $q^i$ is closed inside $U$, thus closed inside $Q_i$.
Therefore, we have shown that there exists a uniform categorical quotient $\phi: U\ra U/\GB$.

We now show that $\phi$ is indeed a uniform geometric quotient: 

\cite[Thm.\;3.(iii)]{seshadri} shows that the image of any $\GB$-stable closed subscheme of $U$ is closed in $U/\GB$, which, combined with the proof of \cite[p.8, Rmk.\;(6)]{git}, shows that $\phi$ is submersive. Since $\GB$ is open over $\Bbase$, by \cite[Rmk.\;2.8.3]{kollar-1997}, we have that $\phi$ is actually universally submersive.

Checking the definition of uniform geometric quotients as in \cite[Def.\;0.6]{git}, we have that to conclude the proof, it suffices to show that for any algebraically closed field $K$ and a point $K$-point $a\in U/\GB(K)$, the fiber $\phi^{-1}(a)$ contains only one $\GB$-orbit:

If the dimension of $\phi^{-1}(a)$ is larger than 1, then by generic flatness, there exists a closed point $u\in \ov{\{a\}}$ such that $\phi^{-1}(u)$ has dimension larger than 1, which contradicts the fact that a closed point in $U$ has a closed orbit in $U$, which is established above.

Therefore, it remains to show that $\phi^{-1}(a)$ is irreducible:

Again, by \cite[Thm.\;3.(ii)]{seshadri}, any two geometric points of $U$ are mapped to the same geometric point of $U/\GB$ by $\phi$ if and only if the closures (inside $U$) of the orbits of the two geometric points intersect nontrivially.
We have seen above that any closed point $u$ in $U$ has a closed orbit in $U$, thus $\phi^{-1}(u)$ is irreducible. 
By \cite[\href{https://stacks.math.columbia.edu/tag/0553}{0553}]{stacks},
we must have that $\phi^{-1}(a)$ is irreducible for any geometric point of $U/\GB$.
We have that finished the proof.
\end{proof}

\begin{rmk}[\bf $U/\GB$ is a tame and good moduli space]
\label{tamegood} 
Let us verify that, by using the terminology as in \cite{alper}, the uniform geometric quotient $U/\GB$ is a tame and good moduli space for the quotient stack $[U/\GB]$:

Since a geometric point in $(X-V)(K)$, where $K$ is an algebraically closed field, has closed orbits in $(X-V)_K$, we have that $X-V$ is the prestable locus for the $\GB$-action on $X$, see \cite[Def.\;10.1]{alper}.
By \cite[Prop.\;11.4]{alper}, there exists a tame and good moduli space $[(X-V)/\GB]\ra (X-V)/\GB$.
By Theorem \ref{ugm proper}.(1), we have that $U$ is open inside $X-V$.
Thus $[U/\GB]$ is an open substack of $[(X-V)/\GB]$, see \cite[Rmk.\;2.3.1]{heinloth}.
Therefore, by \cite[Rmk.\;7.3, Prop.\;7.10]{alper}, we have that $[U/\GB]\ra U/\GB$ is also a tame and good moduli space.
\end{rmk}

\subsection{Universal Closedness of \texorpdfstring{$U/\GB\ra S$}{U/GB-S}}$\;$
\label{proof of uni cl}

In this section, we prove Theorem \ref{ugm proper}.(3) which states that the morphism $U/\GB\ra S$ is universally closed.

\begin{proof}[Proof of Theorem \ref{ugm proper}.(3)]$\;$

Let us start with the following commutative  diagram, where $R$ is a discrete valuation ring, and $L$ is its fraction field
\begin{equation}
\label{kugm}
\xymatrix{
\Spec(L) \ar[r]^-{\xi} \ar[d] & U/\GB \ar[d]\\
\Spec(R) \ar[r] & S.
}
\end{equation}

Let $U^L:=U\times_{U/\GB} \Spec(L)$.
Since $U\ra U/\GB$ is a geometric quotient, by \cite[Def.\;0.6.(ii)]{git}, we have that the orbit morphism $\mu_{\xi}: \Gg_{m,L}\ra U$ factors through a surjective morphism $\mu_{\xi}':\Gg_{m,L}\ra U^L$. Define $\ah_1:=\mu_{\xi}'(1_L)$.
The morphisms $\Gg_{m,L}\xra{\mu_{\xi}'} U^L\ra \Spec(L)$ give field extensions $L\subset k(\ah_1)\subset L$.
Therefore we have that $\ah_1\in U^L(L)$.
We now have the following commutative diagram:
\begin{equation}
\label{LUK}
    \xymatrix{
& U^L \ar[d]_-{\ah_2} \ar[rr] && \Spec(L)\ar@/_/[ll]_{\ah_1} \ar[r] \ar[d]^-{\xi} \ar[dll]_{\eta}& \Spec(R) \ar[d]\\
X & U \ar@{_{(}->}[l] \ar[rr] && U/\GB \ar[r] & S.
    }
\end{equation}

The composition $\ah_1\circ \ah_2=:\eta$ defines an $L$-point $\eta\in U(L)$.
By the properness of $X/S$, we have that $\eta$ induces an $R$-point $r\in X(R)$ filling in the commutative diagram (\ref{LUK}).

Let $Z=Z(r)$ be as defined in Lemma \ref{lemma int Z}.
Recall that $\pi_6: Z\ra \Proj^1_R$ is the composition of iterated $\Gg_{m,R}$-equivariant blow ups 
$p_i:Z_i\ra Z_{i-1}$ at $\Gg_{m,R}$-fixed points of $Z_{i-1}$.
From Lemma \ref{new bl alg}, we see that the centers of all blow ups are in the closed fiber $Z_{i-1}\times_R \kappa$.
Therefore we have that $\pi_6$ is an isomorphism when restricted over the open subscheme $\Proj^1_L\subset \Proj^1_R$. 
Recall that we have the composition of $\Gg_{m,R}$-equivariant morphism $\pi_7: Z\xra{\pi_5} W \xra{\pi_3} X_R \xra{p_X} X$. We have that $\pi_7|_{\pi_6^{-1}(1_L)}=\eta: \Spec(L)\ra X$, and we have the equalities of $L$-points of $X$:
\begin{equation}
\label{pi7pi6 twice}
    \pi_7(\pi_6^{-1}(\infty_L))=\lim_{t\ra \infty} t\cdot \eta, \text{ and } \pi_7(\pi_6^{-1}(0_L))=\lim_{t\ra 0} t\cdot \eta.
\end{equation}

Since $\eta\in U(L)$, we have that $\lim_{t\ra \infty} t\cdot \eta\in V_-(L)$ and $\lim_{t\ra 0} t\cdot \eta\in V_+(L)$. 
We employ  the notation in Lemma \ref{main}. 
By keeping track of what happens to the affine charts containing $0_L$ and $\infty_L$ under each blow up as in Lemma \ref{new bl alg}, we see that $\pi_6^{-1}(\infty_L)$ specializes to the maximal element $z_{n+1}$ with respect to the linear order $\le$ in Definition \ref{def order}, and that $\pi_6^{-1}(0_L)$ specializaes to the minimal element $z_1$ with respect to $\le$. 
By (\ref{pi7pi6 twice}), we have that $\pi_7(z_{n+1})\in \overline{V_-}=V_-$, and that $\pi_7(z_1)\in \overline{V_+}=V_+$. 
Thus there exists a smallest number $1\le j\le n$ so that $\pi_7(z_j)\in V_+$ and $\pi_7(z_{j+1})\in V_-$.
In particular, we have that for any $j'<j$, both $\pi_7(z_{j'})$ and $\pi_7(z_{j'+1})$ are in $V_+$.
Also, for any $j''>j$, we have that $z_{j''}\ge z_{j+1}$, thus both $\pi_7(z_{j''})$ and $\pi_7(z_{j''+1})$ are in $V_-$.
Therefore $j$ is the unique number so that $\pi_7(z_j)\in V_+$ and $\pi_7(z_{j+1})\in V_-$.

Let $\mathcal{E}$ be the irreducible component of $E_n$ that contains both $z_j$ and $z_{j+1}$.
By Lemma \ref{main}, there is an isomorphism $iso: \mathcal{E}\xra{\sim} \Proj^1_{\kappa}$ sending $z_j$ to $0_{\kappa}$ and $z_{j+1}$ to $\infty_{\kappa}$. 
Let $\delta_{\kappa}$ be the $\kappa$-point of $\mathcal{E}$ so that $iso(\delta_{\kappa})=1_{\kappa}$.
By \cite[Lemma 8.3.35.(a)]{liu}, there exists a Weil divisor $\Delta$ of $Z$ that maps surjectively onto $\Spec(R)$ under the projection $Z\ra \Spec(R)$, and contains $\delta_{\kappa}$.
Therefore the fraction field $L'$ of the generic point of $\Delta$ is a finite field extension of $L$.
Let $R'\subset L'$ be a DVR dominating $R$ and let $\delta_{R'}\in Z(R')$, so that $\delta_{R'}$ sends the closed point of $\Spec(R')$ to $\delta_{\kappa}\in E_n(\kappa)$, and the generic point of $\Spec(R')$ to the generic point of $\Delta$.
Let $\delta_{L'}$ be the restriction of $\delta_{R'}$ to the generic point $\Spec(L')\subset \Spec(R')$. 
Since $\pi_7(\delta_{\kappa})\in X(\kappa)$ has 0-limit in $V_+$ and $\infty$-limit in $V_-$, we have that $\pi_7(\delta_{\kappa})\in U(\kappa)$. 
Therefore we have that $\delta_{L'}\in \pi_6^{-1}(\Gg_{m,L})(L')$,
so $\delta_{L'}$ and $\pi_6^{-1}(1_L)$ are in the same $\Gg_{m,L'}$-orbit, thus $\pi_7(\delta_{L'})$ and $\pi_7(\pi_6^{-1}(1_L))=\eta$ are in the same $\Gg_{m,L'}$-orbit in $X$.
Since $U$ is $\GB$-invariant, and  $\eta\in U(L)$, we have that $\pi_7( \delta_{L'})\in U(L')$.
Therefore we have that $\pi_7 (\delta_{R'})\in U(R')$.

We have found $\pi_7( \delta_{L'})\in U(L')$ which is in the same $\Gg_{m,L'}$ orbit as $\eta$, and which can be extended into $\pi_7( \delta_{R'})\in U(R')$. 
In particular, under the quotient morphism $U\ra U/\GB$, we have that $\pi_7( \delta_{L'})$ is sent to $\xi \in U/\Gg_{m,B}(L)$ that appears in our starting diagram (\ref{kugm}).  
 
Consider  the composition $\xi_{R'}: \Spec(R')\xra{\pi_7\circ \delta_{R'}} U \ra U/\Gm$. We  have the following commutative diagram:
\begin{equation}
\label{L1Kxi}
    \xymatrix{
    \Spec(L') \ar[r] \ar[d] & \Spec(L) \ar[r]^-{\xi} & U/\GB\ar[d]\\
    \Spec(R') \ar[urr]^-{\xi_{R'}} \ar[r] & \Spec(R) \ar[r] & S.
    }
\end{equation}

By the valuative criterion in Lemma \ref{valuative uni cl}, we have the universal closedness of the morphism $U/\GB\ra S$.
\end{proof}

\subsection{Separatedness of \texorpdfstring{$U/\GB\ra S$}{U/GmB-S}}
\label{proof of sep}$\;$

In this Section \ref{proof of sep}, we prove Theorem \ref{ugm proper}.(4) which states that $U/\GB\ra S$ is separated.

One word about notation: recall that we have $U_L=U\times_S \Spec(L)$ and $U^L=U\times_{U/\GB} \Spec(L)$.
In this section, we always use upper scripts, such as $U^L$, $U^R$, to denote fiber products over $U/\GB$; while lower scripts, such as $U_L$, $U_R$, $\Proj^1_L$, denote fiber products over $S$ or $B$.

We also employ the notation used  in Section \ref{proof of uni cl}.

Suppose the morphism $\xi: \Spec(L)\ra U/\GB$ in (\ref{kugm}) can be extended to a morphism $\xi_0: \Spec(R)\ra U/\GB$.
The diagram (\ref{LUK}) now becomes
\begin{equation}
\label{lukxi}
    \xymatrix{
   & U^L \ar[d]_-{\ah_2} \ar[rr] && \Spec(L)\ar@/_/[ll]_{\ah_1} \ar[r] \ar[d]^-{\xi} \ar[dll]_{\eta} & \Spec(R) \ar[d] \ar[dl]_-{\xi_0}\\
X & U \ar@{_{(}->}[l] \ar[rr] && U/\GB \ar[r] & S.
    }
\end{equation}

Below we prove the separatedness of $U/\GB\ra S$ by showing that the natural $R'$-point of $U/\GB$ induced by $\xi_0\in (U/\GB)(R)$ coincides with the $R'$-point $\xi_{R'}$ in (\ref{L1Kxi}) in Section \ref{proof of uni cl}.
We have the following diagram where every square is Cartesian:
\begin{equation}
    \xymatrix{
    U^L\ar[r]^-{\beta_3} \ar[d] & U^R \ar[r]^-{\beta_1} \ar[d] & U \ar[d]^-{\phi} & U^R \ar[l]_-{\beta_1} \ar[d] & U^{\kappa}\ar[l]_-{\beta_2} \ar[d]\\
    \Spec(L) \ar[r] & \Spec(R) \ar[r]_-{\xi_0} & U/\GB & \Spec(R)\ar[l]^-{\xi_0} 
    &\Spec(\kappa).\ar[l]_{i}
    }
\end{equation}

\begin{lemma}
\label{irreducible ur}
The fiber products $U^{\kappa}$, $U^L$, and $U^R$ are all irreducible.
\end{lemma}
\begin{proof}
Since a geometric fiber of $\phi$ is a $\GB$-orbit, we have the irreducibility of $U^{\kappa}$ and $U^L$.
We now show that $U^R$ is irreducible.
Since $\GB$ is open over $\Bbase$, by \cite[Rmk.\;2.8.3]{kollar-1997}, we have that $U\ra U/\GB$ is universally submersive, thus the morphism $U^R\ra \Spec(R)$ is submersive. Therefore $U^{\kappa}$ is not open. Thus $U^L$ is not closed. Let $\overline{U^L}$ be the closure of $U^L$ inside $U^R$. If $\overline{U^L}\cap U^{\kappa}$ is a zero dimensional closed subscheme of $U^{\kappa}$, then the morphism $\overline{U^L}\ra \Spec(R)$ violates the upper semicontinuity of dimensions of fibers at the source \cite[Ex III.12.7.2]{hartshorne}. Therefore we must have $\overline{U^L}\cap U^{\kappa}$ is one dimensional. Since $U^{\kappa}$ is irreducible, we have that $\overline{U^L}\cap U^{\kappa}=U^{\kappa}$, and that $U^R=\overline{U^L}$ is irreducible.
\end{proof}

\begin{lemma}
\label{image same}
The set theoretic image of $\beta_1\circ \beta_2: U^{\kappa}\ra U$ is contained in the set theoretic image of the restriction of $\pi_7: Z\ra X$ to the closed subscheme $E_n\subset Z$.
\end{lemma}
\begin{proof}
By Lemma \ref{irreducible ur}, we have that for any closed point $u$ of $U^{\kappa}$, $\dim \mathcal{O}_{U^R,u}=2$.
The proof of \cite[Lemma 3.35]{liu} shows that there exists a finite field extension $L'\supset L$, a discrete valuation ring $R'\subset L'$ dominating $R$, and $\gam_{R'}\in U^R(R')$, so that $\gam_{R'}$ sends the closed points of $\Spec(R')$ to $u$. 

Let $\gam_{L'}$ be the restriction of $\gam_{R'}$ to the generic point $\Spec(L')\subset \Spec(R')$.
Since $U\ra U/\GB$ is a geometric quotient, by \cite[Def.\;0.6.(ii)]{git}, we have that the orbit morphism $\mu_{\eta}:\Gg_{m,L}\ra U$ factors through a surjective morphism $g: \Gg_{m,L}\ra U^L$. Therefore we have a finite field extension $L'\subset L''$ and $\gam_{L''}\in \Gg_{m,L}(L'')$ so that $g(\gam_{L''})=\gam_{L'}\in U^L$.
Let $R''\subset L''$ be a valuation ring that dominates $R'\subset L'$.
We can identify $\Gg_{m,L}$ with $\pi^{-1}_6(\Gg_{m,L})\subset Z$ since $\pi_6$ is an isomorphism over $\Gg_{m,L}\subset \Proj^1_R$.
Since $Z\ra \Spec(R)$ is proper, we can extend $\gam_{L''}\in \Gg_{m,L}(L'')$ to $\gam_{R''}\in Z(R'')$.
Let $\kappa''$ be the residue field of $R''$.
Let $\gam_{\kappa''}$ be the restriction of $\gam_{R''}$ to $\Spec(\kappa'')$.

By construction we have that $\pi_7(\gam_{L''})=\beta_1\circ \beta_2 (\gam_{L'})$.
Since $X/k$ is separated, we have that $\pi_7(\gam_{\kappa''})=\beta_1\circ \beta_2(u)$.
\end{proof}

\begin{proof}[End of Proof of Theorem \ref{ugm proper}.(4)]$\;$

From Lemma \ref{image same}, we see that the set theoretic image of $\beta_1\circ \beta_2: U^{\kappa}\ra U$ lies in the set theoretic image of the restriction of $\pi_7: Z\ra X$ to the unique irreducible component $\mathcal{E}$ of $E_n$ that lies between $z_j$ and $z_{j+1}$.
Therefore, the images of all the $\kappa$-points of $U^{\kappa}$ under $\beta_6\circ \beta_7$ are in the same $\Gg_{m,\kappa}$-orbit as $\pi_7(\delta_{\kappa})$:  recall that $\delta_{\kappa}$ is defined in Section \ref{proof of uni cl} as the closed point $\mathcal{E}$ that is mapped to $1_{\kappa}$ under the natural isomorphism $\mathcal{E}\ra \Proj^1_{\kappa}$.
Furthermore, recall that the image of $\pi_7(\delta_{\kappa})$ under the quotient $U\ra U/\Gm$ is the closed point of the filling 
$\xi_{R'}\in U/\Gm(R')$ selected in the end of Section \ref{proof of uni cl}.
Therefore, we have that the composition $\Spec(R')\ra \Spec(R)\xra{\xi_0} U/\Gm$ agrees with $\delta_{\xi'}: \Spec(R')\ra U/\Gm$ on both the generic and closed points.
Therefore we have a factorization
\begin{equation}
    \xi_{R'}: \Spec(R')\ra \Spec(R) \xra{\xi_0} U/\Gm.
\end{equation}
Since $\xi_{R'}$ is fixed, we have that the lift $\xi_0\in U/\Gm(R)$ of $\xi\in U/\Gm(\xi)$ is unique, thus we have the separatedness of $U/\Gm\ra S$.

The proof of Theorem \ref{ugm proper}.(4), and hence of  the Compactification Theorem I \ref{ugm proper}, is now complete.
\end{proof}

\section{Projectivity of compactifications in Non Abelian Hodge Theory}\la{ss main rz}$\;$

In this section, we apply the compactification/projectivity results of the Projectivity Theorem  III \ref{proj tm}
to first prove Projective Completion Theorems \ref{cpt tm hod} (Hodge), \ref{cpt tm dr} (de Rham)  and  \ref{cpt tm dol} (Dolbeault).  
We focus on the Hodge and Dolbeault picture, since Theorem  \ref{cpt tm dr}  is an immediate consequence
of  the Hodge picture  (take the fiber over  $t=1_{{\mathbb A}^1}$ in Theorem \ref{cpt tm hod}).

\subsection{Some more preparatory Lemmata}$\;$

In this section, we first prove some preparatory Lemmata \ref{lemma frfr}, \ref{frfactor}, \ref{dolfacle} and \ref{dreq} that are stated in \S\ref{rt51}.

Let $J$ be a field of characteristic $p>0$.
Let $C/B$ be the smooth curve as in (\ref{smcurve}).
Recall that we have fixed the rank $r$ and the degree $d$ for the Hodge, Dolbeault, and de Rham moduli spaces.
In the present section, we adopt the following abbreviations:
Let $A$ be the Hithcin base $A(C/\Bbase)$ for the curve $C/\Bbase$.
Let $A^{(\Bbase)}$ be the relative Frobenius twist of $A$.
Let $A'$ be the Hitchin base $A(C^{(B)}/\Bbase)$ for the curve $C^{(\Bbase)}/\Bbase$.

\begin{lemma}[\bf $A^{(\Bbase)}\cong A'$]
\label{lemma frfr}
There exists an isomorphism of $\Bbase$-schemes $A^{(\Bbase)}\cong A'$.
\end{lemma}
\begin{proof}
To find an isomorphism between $A^{(\Bbase)}$ and $A'$ is 
equivalent to find
a natural isomorphism between sheaves of
$\mathcal{O}_{\Bbase}$-algebras:
\begin{equation}\label{frfr}
Sym^{\bullet}((\pi_*\omega_{X/\Bbase})^{\vee})\otimes_{\mathcal{O}_{\Bbase}, fr_B^{\#}}\mathcal{O}_{\Bbase}
\cong 
Sym^{\bullet}((\pi_*^{(B)}\omega_{X^{(B)}/B})^{\vee}).
\end{equation}

Since $X/\Bbase$ is smooth and of relative dimension 1,
working etale locally over $\Bbase$, we can assume that
$\Bbase$ is an affine scheme $\Spec(R)$,
and $X=\Spec(R[x])$.
The coherent sheaf $\omega_{X/\Bbase}$ (resp. $\omega_{X^{(\Bbase)}/\Bbase}$)
is the rank 1 free $R[x]$-module (resp. $R[x\otimes 1]$-module)
with a generator $dx$ (resp. $d(x\otimes 1)$).
Let $y$ (resp. $y\otimes 1$) be the element in $Hom_R(R[x],R)$ (resp. $Hom_R(R[x\otimes 1],R)$)
that sends $x$ (resp. $x\otimes 1$) to $1\in R$.
The left hand side of (\ref{frfr}) corresponds to the $R$-algebra
$R[y,\partial_x]\otimes_{R,fr_R}R$, while
the right hand side of (\ref{frfr}) corresponds to the $R$-algebra
$R[y\otimes 1,\partial_{x\otimes 1}]$, 
The assignment $f\partial_x \otimes 1\mapsto (f\otimes 1)\partial_{x\otimes 1}$ induces an isomorphism of $R$-algebras from the left hand side to the right hand side of (\ref{frfr}).
 \end{proof}
 
\begin{rmk}[$Fr_A\text{\say{=}}\sig_{\Bbase}^*$]\label{frfactor}
 Assuming the notation in Lemma \ref{lemma frfr} and its proof, we see that the comorphisms of the $\Bbase$-morphisms $A\ra A^{(\Bbase)} \ra A'$ are determined by the assignments $(f\otimes 1)\partial_{x\otimes 1}\mapsto f\partial_x\otimes 1\mapsto (f\partial_x)^p\in R[y,\partial_x]$.
 We can then derive another description of the compositum $F: A\ra A'$:
 
 A $B$-point of $A$ is a linear combination of terms of the form $ r_i x^j(dx)^k$ with $r_i\in R$.
We have that 
\[((y\otimes 1)^j\partial_{x\otimes 1}^k)\Big(F(B)\big(r_i x^j (dx)^k\big)\Big)=\Big((y^j\partial_x^k)\big(r_ix^j(dx)^k\big)\Big)^p=r_i^p.\]

Therefore we see that $F(B)$ sends $dx$ to $dx\otimes 1$, $x$ to $x\otimes 1$, and $r_i$ to $r_i^p$.
Therefore, the function $F(B):A(B)\ra A'(B)$ is induced by the pull back
\begin{align*}
    & H^0(B\times_B C, pr_C^*\omega_{C/\Bbase}) &&(=H^0(C,\omega_{C/\Bbase}))\\
    \xra{(fr_B^*,\sig_B^*)}
    & H^0(B\times_B C^{(\Bbase)}, pr_{C^{(\Bbase)}}^*\omega_{C^{(\Bbase)}/\Bbase}) &&(=H^0(C^{(\Bbase)}, \omega_{C^{(\Bbase)}/\Bbase})).
\end{align*}
 
Similarly, one can show that, for any $\Bbase$-scheme $T$, the morphism $F(T): A(T)\ra A'(T)$ is induced by 
 the pull back $(\sig_{\Bbase}, fr_T)^*$.
 
 Note that by \cite[\S2.3]{la-pa}, we have the following commutative diagram:
 
 \begin{equation}
 \xymatrix{
 C^{(\Bbase)}\times_B T \ar@/^1pc/[rr]^-{(\sig_{\Bbase},fr_T)} \ar[r]_-{\simeq} \ar[dr]_-{pr_T}
 & (C\times_B T)^{(T)} \ar[r]_-{\sig_T} \ar[d] & C\times_B T \ar[d]\\
 & T \ar[r]_-{fr_T} & T.
 }
 \end{equation}
 
 Therefore, up to a natural identification $C^{(\Bbase)}\times_B T\cong (C\times_B T)^{(T)}$, the morphism 
 $F(T)$ is induced by $\sig_T^*$.
 \end{rmk}

 \begin{lemma}[\bf Factorization of Hodge-Hitchin Morphism over $0_{A'}$]
 \label{dolfacle}
 Let $A^p$ be the $\Bbase$-scheme that is the total space of the locally free $\mathcal{O}_{\Bbase}$-module
$\bigoplus_{i=1}^r\pi_*\omega_{X/\Bbase}^{\otimes ip}$.
 There exists the following commutative diagram:
 \beq\la{zbgdg}
\xymatrix{
M_{Hod}(C/\Bbase)_{0_{A'}}   \ar[rr]^-{h_{Hod,0_{A'}}} && A'  
 \ar[rr]^-{Fr_{C/\Bbase}^*}& &    A^p
\\
M_{Dol}(C/\Bbase) \ar[u]  \ar[rr]_-{h_{Dol}}&& A \ar[urr]_-{fr_C^*} \ar[u]^-{\sig_{\Bbase}^*``="Fr_A}. & &
}
\eeq
 \end{lemma}
 \begin{proof}
Since the $p$-curvature of a Higgs bundle $\phi$ is $\phi^p$, we have that the outer 5-gon in (\ref{zbgdg})
is commutative.
Since $fr_C=\sig_{\Bbase}\circ Fr_{C/\Bbase}$, by Remark \ref{frfactor}, we see that the triangle in (\ref{zbgdg}) is commutative.
Furthermore, the morphism $Fr_{C/\Bbase}^*$ is a monomorphism.
Therefore, the bottom left square of (\ref{zbgdg}) is commutative.
 \end{proof}

\begin{rmk}
The commutativity of diagram (\ref{zbgdg}) shows that the isomorphism $A^{(\Bbase)}\cong A'$ is $\GB$-equivariant, since all the other arrows involved in the left square of (\ref{zbgdg}) are $\GB$-equivariant. 
\end{rmk}

\begin{lemma}[\bf Trivialization of Hodge-Hitchin Morphism over $\GB$]
\label{dreq}
The natural $\Bbase$-morphism $M_{dR}(C/\Bbase) \to M_{Hod}(C/\Bbase)\times_{\mathbb A^1_{\Bbase}} 1_{\Bbase}$ is an isomorphism.
There exists a natural isomorphism of $\Bbase$-schemes $M_{dR}(C/\Bbase)\times_{\Bbase} \GB\cong M_{Hod}(C/\Bbase)\times_{\mathbb A^1_{\Bbase}}\GB$.
Furthermore, we have the following commutative commutative diagram:
\begin{equation}
\label{diagtrihod}
    \xymatrix{
    M_{Hod}(C)\times_{\mathbb A^1}\GB \ar[d]_-{h_{Hod,\GB}} &
    M_{dR}(C)\times_{\Bbase} \GB \ar[l]_-{\simeq} \ar[d]^-{h_{dR}\times 1_{\GB}}\\
    A'\times \GB & A'\times \GB\ar[l]_{\simeq}
    }
\end{equation}
\end{lemma}
\begin{proof}
Since $M_{Hod}$ is uniformly corepresenting, we have that the fiber product $M_{Hod}\times_{\mathbb A^1}\GB$ is corepresenting the functor of $t$-connections with invertible $t$.
Therefore, the morphism  $((E,\nabla),t)\mapsto (E,t\nabla))$ defines an isomorphism between the functors that are corepresented by $M_{dR}\times_{\Bbase}\GB$ and $M_{Hod}\times_{\mathbb A^1}\GB$, thus an isomorphism between the corepresenting schemes.
Since $f$ is also an isomorphism of $\GB$-schemes, we have an isomorphism $M_{dR}\cong M_{Hod}\times_{\mathbb A^1} 1_{\Bbase}$.
When $J$ is a field of characteristic $p>0$, the bottom isomorphism in (\ref{diagtrihod}) is given by $(a_i,t)\mapsto (t^{ip} a_i, t)$, with $a_i\in \mathbb A_i'$ (recall that $\mathbb A_i'$ is the direct factor of $A(C^{(\Bbase)}/\Bbase)$ corresponding to the locally free sheaf $\pi_*^{(\Bbase)}\omega_{X^{(\Bbase)}/\Bbase}^{\otimes i}$).
\end{proof}

\subsection{Proof of Theorems \ref{comphodj} and \ref{cpt tm hod}}\la{bnbn}$\;$

\begin{proof}[Proof of Theorem \ref{comphodj}]
We have the basic $\GB$-equivariant diagram with Cartesian square (where the subscript $B$ is dropped) from \ci[\S3.1]{fe-ma}. This construction  already appears in \ci[Proof of Lemma 6.1]{ha-th}.
\beq\la{act}
\xymatrix{
Z:= M_{Hod} \times_{\mathbb A^1} \mathbb A^2 \ar[r] \ar[d] \ar@/_2pc/[dd]_-{\tau'} & M_{Hod} \ar[d]^-\tau  & &
\\
\mathbb A^2_{x,y} \ar[r] \ar[d] & \mathbb A^1_\lambda & (x,y) \ar@{|->}[r] \ar@{|->}[d] & \lambda=xy
\\
S:=\mathbb A^1_x &  &x, &
}
\eeq
where the $\GB$ action on $\mathbb A^2_{x,y}$ is defined by setting $t(x,y):= (x,ty)$, the $\GB$ action  on $\mathbb A^1_\lambda$ is the usual dilation
$t\cdot \lambda:= t \lambda,$ and the $\GB$  action on $\mathbb A^1_x$ is trivial. 

We would like to apply the Compactification Theorem II \ref{ugm proper 2} to $Z/S$ in (\ref{act}).
Below we show that the assumptions in Theorem (\ref{ugm proper 2}), i.e. that zero limits of Zariski points exist in $Z$, and that the fixed point locus in $Z$ is proper over $S$, are satisfied:

A Zariski point $z\in Z(k)$, where $k$ is a field, can be represented by a pair $((E,\nabla), (x,y))$ on $C_k$, where $(E, \nabla)$ is semistable as a vector bundle with an $xy$-connection. 
The $\Gg_{m,k}$-orbit of $z$ is then represented by $((E,t\nabla), (x,ty))$.
We can naturally extend this $\Gg_{m,k}$-orbit to obtain an $\mathbb{ A}^1_k$-family with the element over $0_k\in \mathbb{ A}^1_k$ being $((E,\phi), (x,0))$, where $(E,\phi)$ is a possibly non-stable Higgs bundle.
By Langer's Langton-type result \cite[Th.\;5.1]{la-2014}, we can change the $0_k$-fiber $(E,\phi)$ to a semistable $(E',\phi')$ so that we obtain a morphism 
$\mathbb{A}^1_k\ra Z$ which maps $t\in \Gg_{m,k}(k)$ to $((E,t\nabla), (x,ty))$ and maps $0_k$ to $(E', \phi')$.
Therefore, we have that $Z$ has all its zero limits.

The $\GB$ fixed locus on $M_{Hod}$ is contained in the fiber  $M_{Hod,0_{\mathbb A^1_{\lambda}}}$ of $\tau$ over the origin $0_{\mathbb A^1_{\lambda}} \in \mathbb A^1_{\lambda}$. 
Since the natural morphism $M_{Dol}\ra M_{Hod,0_{\mathbb A^1_{\lambda}}}$ is $\GB$-equivariant and bijective on geometric points, we have that the set underlying $\GB$ fixed locus is naturally identified with the $\GB$ fixed locus of $M_{Dol}$ inside $h_{Dol}^{-1}(o_{A'})$.

We now show that $h^{-1}_{Dol}(o_{A'})$ is proper over $\Bbase$:
Using the valuative criterion for properness, we are reduced to the case where $\Bbase$ is a 
discrete valuation ring, but then 
the Langton-type argument in the proof of \ci[Thm. I.3]{fa}
goes through and gives the properness of $h^{-1}_{Dol}(o_{A'})$ over $\Bbase$.
Since the $\Bbase$-morphism $M_{Dol}\ra M_{Hod,0}$ is bijective on geometric points,
using the valuative criterion Lemma \ref{valuative uni cl} (taking $L$ to be algebraic closure of $K$ so that we can lift the $L$ point from $M_{Hod,0_{\mathbb{A}^1}}$ to $M_{Dol}$), we have that the
$\GB$-fixed locus of $M_{Hod}$ is also proper over $\Bbase$. 

The following can then be verified:

(1)
The  $\GB$ fixed point set in $Z$ is proper over $S$;

(2)
The complement $U$ in $Z$ of the set of points in $Z$ admitting infinity limits, 
is $Z$ minus the $x$-axis times the closed subset of $M_{Hod}$ that is universally homeomorphic to 
the proper fiber $h_{Dol}^{-1}(o_{A'})$;

(3)
The open subvariety  $T \subseteq Z$  obtained by removing the preimage of the origin via the projection onto the $y$-axis,
endowed with the projection to this puncture $y$-axis  is $\GB$-equivariantly isomorphic to $M_{Hod} \times \GB$ (cf. (\ref{diagtrihod})).

Apply the Projectivity Theorem III \ref{proj tm} to this situation. The Projective Completion of $\tau: M_{Hod} \to \mathbb A^1$
Theorem \ref{comphodj}
follows once we set $\ov{M_{Hod}}:= U/\GB$ etc. 

To finish the proof, we simply need to observe that:

(i)
 $M_{Hod}$ admits a natural open immersion into $\ov{M_{Hod}}$ by property
(3) above.  

(ii)
$U/\GB$ admits a natural  $\GB$ action compatible
with the open immersion (i); this action is already present on $Z$ and on  $U$ by letting $\GB$ act by the standard weight one dilation $t\cdot x:=tx$ on the $x$ coordinate.

(iii)
We also need to endow $S:=\mathbb A^1_x$ (which originally had the trivial action, so we could  arrive to $U/\GB \to S$ via the Compactification Theorem II \ref{ugm proper 2})
with the standard weight one dilation action, so that, now, $\ov{\tau}$, and in fact the whole diagram (\ref{eq001}), is $\GB$-equivariant.  The proof of Theorem \ref{comphodj} is thus completed.
\end{proof}

\begin{proof}[Proof of Theorem \ref{cpt tm hod}]
We would like to apply the Compactification Theorem III \ref{proj tm}.
We define the $Z$, $Z'$, and $S$ in Theorem \ref{proj tm} in the following way:

We augment the diagram (\ref{act}) by inserting the Hodge-Hitchin morphism. 
We obtain the $\GB$-equivariant commutative diagram with Cartesian squares
\beq\la{act hh}
\xymatrix{
& Z:= M_{Hod} \times_{\mathbb A^1} \mathbb A^2 \ar[r] \ar[d]^-{h_{Hod}'}  \ar@/_/[ddl]_-{\tau'} & M_{Hod} \ar[d]^-{h_{Hod}}  
\ar@/^3pc/[dd]^-{\tau}
\\
& Z':=A' \times_\Bbase \mathbb A^2_{x,y} \ar[d] \ar[r]  & A' \times_\Bbase \mathbb A^1_\lambda \ar[d]
\\
S:=\mathbb A^1_x &  \mathbb A^2_{x,y} \ar[r] \ar[l] & \mathbb A^1_\lambda.
}
\eeq

To check that the assumptions of Theorem \ref{proj tm}, we note that,
firstly,
the assumptions in Theorem \ref{proj tm} that are only about $Z$ and $U$ are checked in the proof of Theorem \ref{comphodj};
secondly,
other assumptions involving $Z'$ and $U'$ are easily checked to be satisfied.
The existence of the commutative diagram (\ref{eq01}), the item (1) of Theorem \ref{cpt tm hod},
and the projectivity of $\ov{h_{Hod}}$ and $\ov{\tau_{Hod}}$
then follows from the application of Theorem \ref{proj tm}.

Since $\ov{h_{Hod}}$ maps the boundary of $\ov{M_{Hod}}$ to the boundary $\ov{A'}\times \mathbb A^1_{\Bbase}$, we have that the Hodge-Hitchin morphism $h_{Hod}$ is also projective. Thus we have the item (2) of Theorem \ref{proj tm}. 

By inspecting the construction of $Z'$,  it is clear that  we obtain the  weighted projective space $\mathbb P (1,1 \cdot p,2p,\ldots, rp).$
This latter coincides with  
$\mathbb P (1,1,2,\ldots, r)$ in view of the fact that  we  can replace $ip$ by $i$;  see Remark \ref{wz}.
The proof of Theorem \ref{cpt tm hod} is thus completed.

\end{proof}

\subsection{Proof of Theorems \ref{Jcpt tm dol} and \ref{cpt tm dol}}\la{qnbn}$\;$

\begin{proof}[Proof of Theorem \ref{Jcpt tm dol}]
We would like to apply Compactification Theorem III \ref{proj tm} as in the proof of Theorem \ref{cpt tm hod} in \S\ref{bnbn}.

To obtain a similar diagram as (\ref{act hh}).
we substitute the right column of (\ref{act hh}) by the morphisms 
$M_{Dol}\xra{h_{Dol}'} A\times_{\Bbase} \mathbb A^1 \xra{pr}\mathbb A^1$,
where $h_{Dol}'$ denotes the Hitchin morphism $h_{Dol}: M_{Dol}\to A$ followed by the closed embedding $A=A\times \{0_{\mathbb A^1}\}
\to A \times \mathbb A^1.$ 
We can then add the corresponding analogue of the left column in (\ref{act hh}) by the construction in (\ref{act}).
Then the arguments in \S\ref{bnbn} applies, \textit{mutatis mutandis}, and finishes the proof of all the statements in Theorem \ref{Jcpt tm dol} about the left half of the diagram
(\ref{ahmm}).

For the remaining statements in Theorem \ref{Jcpt tm dol}, we need two variations of the diagram (\ref{act}):
For the first variation, we first replace the right column of (\ref{act}) by the constant $\Bbase$-morphism $M_{Dol}\ra \mathbb A^1_{\Bbase}$ that sends $M_{Dol}$ to $0_{\Bbase}\subset \mathbb A^1_{\Bbase}$, 
and then construct the left column as in (\ref{act})-- this gives us the compactification of $\ov{M_{Hod}}$ above;
For the second variation, we take the base change of the whole diagram (\ref{act}) via the inclusion $0_{\lam}\hra \mathbb A^1_{\lam}$-- this gives us the compactification $\ov{M_{Hod,0_{\Bbase}}}$ and shows that 
the fiber $(\ov{M_{Hod}})_{0_{\Bbase}}$ of the compactification of $M_{Hod}$ constructed in Theorem \ref{comphodj} is the compactification of $M_{Hod,0_{\Bbase}}$.
The natural morphism $M_{Dol}\ra M_{Hod,0_{\Bbase}}$ then induces a morphism of the two variations of diagram (\ref{act}), thus a morphism $\ov{M_{Dol}}\ra \ov{M_{Hod,0_{\Bbase}}}$.
The remaining statements in Theorem \ref{Jcpt tm dol} can then be checked routinely.
\end{proof}

\begin{proof}[Proof of Theorem \ref{cpt tm dol}]
We assume the notation in the proof of Theorem \ref{Jcpt tm dol}.
Consider the commutative diagram of $\GB$-equivariant morphisms
\beq\la{mm11}
\xymatrix{
M_{Dol} \ar[rr] \ar[d]^-{h_{Dol}'} & & M_{Hod} \ar[d]^-{h_{Hod}}
\\
A \times_B \mathbb A^1\ar[rr]^-{Fr_A \times Id_{\mathbb A^1}} \ar[rd]_-{pr} && A'\times_B \mathbb A^1 \ar[ld]^-{pr}
\\
& \mathbb A^1, &
}
\eeq
We repeat essentially verbatim the arguments in \S\ref{bnbn}, by applying the construction (\ref{act}) and augmenting it
using (\ref{mm11}), the same way we used (\ref{act hh}). Of course, we need the evident ``multiple morphisms" version of the Compactification Theorem II  \ref{ugm proper 2} and of the Projectivity Theorem III \ref{proj tm}.
The items (1)-(3) of Theorem \ref{cpt tm dol} then follow.
(Let us remark that, in view of the factorization $M_{Dol}\to M_{Hod,0_{\mathbb A^1}} \to A',$  the projectivity of $M_{Dol}\to M_{Hod,0_{\mathbb A^1}}$ follows immediately from the projectivity of the compositum
$M_{Dol}\to  A \to A'$.)

For the item (4) of Theorem \ref{cpt tm dol}, note that $c: M_{Dol}\to M_{Hod,0_{\mathbb A^1}}$ is a universally closed bijection. If the rank $r$ and degree $d$ are  coprime, then the morphism $c$ is an isomorphism by the universal corepresentability
property of  the  Hodge moduli space; see Remark \ref{jj}.
We are done if we can show that $c$ is indeed a universal bijection, as universally closed universal bijections are universal homeomorphisms.
Given any Zariski point $x$ of $M_{Dol}$, we would like to show that the field extension $\kappa(x)\supset \kappa(c(x))$ is purely inseparable.
By taking the closure of $x$ and $c(x)$, and shrinking $\ov{c(x)}$ if necessary, we may assume that $\ov{x}$ and $\ov{c(x)}$ are both normal, and $c|_{\ov{x}}$ is finite.
Let $\tilde{x}$ be the normalization of $\ov{x}$ in the separable closure of $\kappa(c(x))$ inside $\kappa(x)$.
We have that $\tilde{x}\ra \ov{c(x)}$ is generically etale and injective.
Since $J$ is algebraically closed, and here is the only place where we use this assumption on $J$,
we have that $\tilde{x}\ra \ov{c(x)}$ is an isomorphism over an open and dense subset of $\ov{c(x)}$.
Therefore, we have that $\kappa(\tilde{x})=\kappa(c(x))$, thus the field extension $\kappa(x)\supset \kappa(c(x))$ is purely inseparable. We have thus proved item (4). 

For the item (5) of Theorem \ref{cpt tm dol}, note that by construction, we obtain the morphism 
\beq\la{kij}
\xymatrix{
\ov{A}=\mathbb P (1,1,2,\ldots,p) \ar[r] & \ov{A'}=\mathbb P (1,p,2p, \ldots, rp)=P (1,1,2,\ldots,r)
=\ov{A}.
}
\eeq
Note that the Frobenius twist of $\ov{A}$ is $\Proj(p,p,2p,...,rp)=\Proj(1,1,2,...,p)=\ov{A}=\ov{A'}.$
Since the restriction of $\ov{A}\ra \ov{A'}$ to $A$ is the relative Frobenius $Fr_A$, we have that
the morphism (\ref{kij})  $\ov{A}\ra \ov{A'}$ is also the relative Frobenius $Fr_{\ov{A}}$ for $\mathbb P (1,1,2,\ldots, r)$.

The compactification and $M_{Dol}$ Theorem \ref{cpt tm dol} follows. 

\end{proof}

\section{Appendix: smooth moduli and specialization}\la{appdx}$\;$

\subsection{Introduction to the Appendix}\la{intro appdx}$\;$

{\bf The paper \ci{de-2021}} is devoted to develop a formalism for the  specialization morphism, when it exists,  as a perverse Leray filtered morphism for a family of morphisms $f:X\to Y$ over a base curve $S$; the ground field is the
one  of the  complex numbers.
While \S1,2,3 of loc.cit. are rather general, \S4 of loc.cit.  is devoted to applying the formalism when the morphism $f$ is   the Hitchin morphism  $M_{Dol}(C/S) \to A(C/S) \to S$ associated with  a  smooth curve
 (\S\ref{rt51})  $C/S$.
 The main result is \ci[Tm 4.4.2]{de-2021}, to the effect  that the specialization morphisms for this family are 
defined and are filtered isomorphisms.  Another relevant result is \ci[Lm 4.3.3]{de-2021}, to the effect that 
$\phi (v_* \rat_{M_{Dol}})=0$ for  the vanishing cycle
functor applied to the direct image complex with respect to the structural morphism $v:M_{Dol}(C/S)\to S.$
Note that neither statement is a priori clear, since the morphism $v$ is not proper.

{\bf In the paper \ci{de-szh naht}}, we use the generalization of \ci[Tm. 4.4.2, Lm 4.3.3]{de-2021} to the cases where the base is a complete 
strictly Henselian DVR $S$, and the morphisms $f$ are: 1) the Hodge-Hitchin morphism (\ref{eqdefh})
$M_{Hod}(C/\field) \to \mathbb  A(\FC)\times A^1_\field \to \mathbb A^1_\field$, after base change to the appropriate local ring $S$ at the origin of $\mathbb A^1_\field$; 2)
the Hitchin morphism (\ref{hitz mo})  $M_{Dol}(C/S) \to A(C/S)\to S$.
These moduli space are with respect to certain coprimality conditions  on rank, degree and characteristic of the ground field. These conditions turn out to imply the smoothenss of these moduli spaces
and that they  universally corepresent the appropriate functors
(so that taking the fibers over $S$, one gets the expected moduli space). 
The  desired generalization of these results  is not a simple matter of  routine and has  as starting point the compactification results stated in  \S\ref{rt51} and proved in \S\ref{bnbn}.

{\bf The purpose of this appendix} is to tie in the compactification results of this paper with \ci[\S4]{de-2021} by providing the necessary background so that we can prove \ci[Tm. 4.4.2, Lm. 4.3.3]{de-2021} for the Hodge and Dolbeault moduli spaces for smooth curves $C/S$ (\ref{rt51}) over a complete strictly Henselian DVR $S$ as in the previous paragraph, so that we can use these results in \ci{de-szh naht}.

{\bf Brief summary of \ci[\S4]{de-2021}.}
\ci[\S4]{de-2021} uses the compactification constructed in
\ci[Tm. 3.1.1 and (14)]{de cpt}; this construction uses the same kind of  
quotient by $\Gm$ technique used in this paper.  The outcome of the construction is summarized in the diagrams 
\ci[(70), (72)]{de-2021}: the compactification of the  moduli space is denoted by an open immersion
 $X^o \subseteq X$ with boundary $Z$ (i.e. a triple $(Z,X,X^o)$);  the  compactification is obtained by taking the quotient by $\Gm$ of a suitable  triple $(\ms{Z},\ms{X},\ms{X}^o)$.
\ci[\S4]{de-2021} uses in an essential way,
the topological local triviality, due to C. Simpson, of the Dolbeault moduli spaces  over the base $S$ (all over the complex numbers). This implies  the same kind of local triviality of the triple  $(\ms{Z},\ms{X},\ms{X}^o)$. 
In turn, this implies the vanishing $\phi(\rat_{\ms{X}})= \phi(\rat_{\ms{Z}})=0$ of the vanishing cycles
\underline{before} taking the quotient by $\Gm.$ Due to the local product structure (\ci[Lm. 4.3.1]{de-2021}), 
one also has, \underline{before} taking the quotient, that 
$a^!{\oql}_{\ms{Z}}= {\oql}_{\ms{Z}}[-2],$ where $a: \ms{Z}
\to \ms{X}$ is the closed embedding.
The key point in proving \ci[Tm.4.4.2, Lm. 4.3.3]{de-2021} is to descend,
along the quotient by $\Gm$,  the  vanishings and identities  above from the triple
$(\ms{Z}, \ms{X}, \ms{X}^o)$  to the triple $(Z,X,X^o).$

{\bf Plan to fulfill the purpose.}
In order to achieve the desired purpose stated above, the plan is to follow the path traced in \ci[\S4]{de-2021}
over the complex numbers and make the necessary adjustments along the way when working over the DVR $S$.

{\bf The suitable  $\oql$-adic formalism in \S\ref{recdvr}.}
First of all, we need a suitable formalism of $\oql$-adic constructible sheaves on separated schemes of finite type over the complete strictly Henselian  DVR $S$. This is the content of \S\ref{recdvr}, where 
we collect some results in the literature to provide a linear exposition of the eight functor formalism  and of perverse sheaves  for $\oql$-adic constructible
complexes on  separated schemes of finite type over an excellent DVR. 
With the formalism of \S\ref{recdvr} at our disposal, we recover virtually  the whole of the machinery in  \ci[\S1,2,3]{de-2021}, and we are ready to tackle the purpose of this appendix.

{\bf The compactifictions we use.}
We use the compactifications of Hodge and Dolbeault moduli given  in \S\ref{rt51} by  Theorems \ref{cpt tm hod}
and \ref{cpt tm dol}. The key construction is summarized by diagrams (\ref{act}) and (\ref{act hh}).
Warning on notation concerning case of the  Hodge moduli space:  i) what has been denoted by $\ms{X}$ above in the brief summary is a suitable open subset $U$ of what is denoted $Z$ in (\ref{act hh}) (cf. the proof of Theorem
\ref{comphodj} in \S\ref{bnbn}) ; ii) what is denoted by $\ms{Z}$ above, is the closed  subset of $U= \ms{X}$
given by the preimage of the $x$-axis with the fiber of the Hitchin morphism over the origin (nipotent cone $N_{Dol}$) 
removed from every copy of $M_{Dol}$ over the points of the $x$-axis; then $\ms{Z}$ is isomorphic to
the product ($x$-axis)$\times (M_{Dol} \setminus N_{Dol}).$ Despite the spaces being singular, the closed embedding $\ms{Z} \subseteq \ms{X}$ is regular of codimension one.

In general, the Hodge and Dolbeault moduli spaces are not regular.
We leave the issues of smoothness over the appropriate base,  and of  universal corepresentability
 for the Hodge and Dolbeault moduli spaces to \ci{de-szh naht}  (they follow from suitable coprimality assumptions), Next, we tackle all the remaining issues that arise in connection to generalizing  \ci[Th. 4.4.2, Lm. 4.3.3]{de-2021}
 from $\comp$ to  a complete strictly
Henselian DVR $S$.

{\bf List of remaining technical issues in \S\ref{a cpt sp}.}
At this juncture, a close inspection of \ci{de-2021} reveals that the only issues that arise when trying to fulfill
the purpose of this appendix, i.e.  generalize  \ci[Lm. 4.3.3 and Tm. 4.4.2]{de-2021} to these compactifications discussed above over the DVR $S$,  are the aforementioned smoothness and universal corepresentability (dealt with in \ci{de-szh naht}),
and a small list of technical facts that need to be suitably generalized when replacing  the base ring $\comp$ with an algebraically closed field, or with a complete strictly Henselian DVR. Dealing with this list, is the content  of 
\S\ref{a cpt sp}.

\subsection{Rectified perverse $t$-structure over a DVR}\la{recdvr}$\;$

We collect and complement references in the literature, so that one can work with nearby/vanishing cycles in the context of $\oql$-adic coefficients and middle perversity $t$-structures. The key ingredient is O. Gabber
rectified middle perversity $t$-structure for schemes of finite type over a DVR.

{\bf The trait.}
 Let $(S,s,\eta)$ be  trait, i.e.  the spectrum of a DVR (discrete valuation ring),
with closed point $i: s \to S$, and with generic open point $j: \eta \to S.$   With the exception
of (\ref{psph}), we work with schemes $f: X \to S$ that are separated and of finite type over $S$ and with $S$-morphisms that are separated and of finite type.
The special  closed fiber is denoted  $X_s$ and the generic open fiber $X_\eta.$ Fix a prime number $\ell$ that is invertible in $S.$

 {\bf The constructible  $\oql$-adic derived category.}
The trait $S$ is Noetherian, regular  and of dimension one. In particular, we have access to  the finiteness results in 
\ci[Th. Finitude]{sga45}.  Let $X/S$ be as above.
Let $D^b_c(X,\oql)$ be the $\oql$-constructible derived category, whose objects we call
(constructible) complexes; see \ci[Thm. 6.3]{ek}; 
 see also \ci[\S5]{fa}. It is endowed with a natural $t$-structure, with heart the abelian category 
 of $\oql$-constructible sheaves on $X.$

{\bf The formalism of functors.}
As $X/S$ varies,  with $S$ fixed, the categories $D^b_c(X,\oql)$ enjoy the usual formalism of the eight (``derived") functors 
\[
(f_!, f^!), (f^*, f_*), (\otimes, {\rm Hom}), \psi, \phi,
\]
with the usual adjunction relations among them --in each parenthesis $(A,B)$ above, $A$ is left adjoint to $B$--, as well as duality exchanges --such as $Df_!=f_*D, Df^!=f^*D, D \psi [-1] = \psi [-1] D$ (for the last one, use  \ci[Thm.4.6]{illusie}, and $\psi = \psi_\eta \circ j^*$).   
These functors are exact (additive, commute with translations, preserve distinguished triangles).
See \ci[Th. 6.3]{ek}, for a  partial list of the properties concerning
$f_!, f^!, f^*, f_*, \otimes, {\rm Hom}$.  The functors $D$, $\psi$ and $\phi$ are discussed below.

{\bf Duality.} The duality functor $D$ used above is introduced as follows. We call the constructible complex $K_S:= {\oql}_S [2](1)$ the dualizing sheaf of $S$; see \ci[Tm. 6.3.(iii)]{ek}.
The object $K_{X/S}:=f^!K_S$
is a dualizing sheaf on $X$ relative to $S$.  
We denote by $D:={\rm Hom} (-, K_{X/S}): D^b_c(X, \oql) \to D^b_c(X,\oql)$  the
corresponding dualizing contravariant functor.  Note that $K_{s/S} = i^!K_S = K_{s/s}={\oql}_s$, i.e. the dualizing sheaf of $s$ relative to $S$ coincides with the usual dualizing sheaf of $s$. On the other hand, $K_{\eta/S} = j^! K_S =  {\oql}_\eta [2](1)
\neq {\oql}_\eta =K_{\eta/\eta}.$

{\bf  Nearby  and vanishing cycles functors.} For the nearby and vanishing cycles functors
 $\psi$ and $\phi$ see \ci[7.1 I, 7.2 XIII]{sga7}, as well as
\ci{illusie, illusie2}, and references therein. Note that what we denote by $\phi$ here, is denoted by $\phi [-1]$ in loc.cit.; in particular, the usual   distinguished triangle of functors appears here as (\ref{dtf})   $i^* \to \psi \to \phi [1] \rightsquigarrow.$

When using the functors $\psi, \phi$ we assume in addition that the trait $S=S^h$  is Henselian.
Choose a separable closure of the residue field. Form the associate strict Henselianization
$S_{(\ov{s})}$ of $S$ at $s.$  Choose a separable closure
of the fraction field of the strict Henselianization.
After base change, we obtain 
 natural  morphisms of $S$-schemes ($\e$ and $\ov{j}$ are not of finite type)
\beq\la{psph}
\xymatrix{
X_{\ov{s}} \ar[r]^-{\ov{i}} & X_{S_{(\ov{s})}} & X_{\ov{\eta}} \ar[l]_-{\ov{j}} \ar[r]^-\e
& X_{\eta} \ar[r]^-j & X.
}
\eeq
We define the nearby cycles functor by setting  
\beq\la{psipp}
\psi := \ov{i}^* \ov{j}_* \e^* j^* : D^b_c(X, \oql),  \lorw D^b_c (X_{\ov{s}}, \oql).
\eeq
The constructibility assertion here is from  \ci[Th. Finitude, Tm. 3.2]{sga45} and 
\ci[p.45 top]{illusie} (this is what is needed in \ci{ek} to land in the constructible derived category).
We also have the  more classical nearby cycles  functor $\psi_\eta: D^b_c(X_\eta, \oql) \to D^b_c (X_{\ov{s}}, \oql)$, obtained by setting $\psi_\eta:= \ov{i}^* \ov{j}_* \e^*.$ Clearly, 
$\psi = \psi_\eta \circ j^*,$ i.e. $\psi (F)$ depends only on the restriction $j^*F.$
Let $\nu: X_{S_{(\ov{s})}} \to X$ be the natural projection morphism. By adjunction, we have a natural
morphism $\ov{i}^* \nu^* \to \psi$, which we simply denote by $i^* \to \psi$ (this is literally correct if $S$ is strictly Henselian). The cone of this morphism is in fact functorial (cf. \ci[7.2 XIII, (1.4.2.2, 2.1.2.4)]{sga7}, and we denote it by $\phi[1]$.
We have a distinguished triangle of functors 
\beq\la{dtf}
i^* \to \psi \to \phi[1] \rightsquigarrow .
\eeq

For a partial  list of the properties concerning $\psi$ and $\phi$, see \ci[\S2.1]{de-2021}. See also \ci[\S4]{illusie}
(in particular, see Th. 4.7 for the compatibility with cup products).

{\bf The rectified middle perversity $t$-structure.}
We go back to the case where $S$ is a trait (not necessarily Henselian).
For $X/S$  separated and of finite type, 
the   category $D^b_c(X, \oql)$  is endowed with  O. Gabber's  ``rectified"  middle-perversity $t$-structure \ci[\S4, \S2]{illusie, illusie2}
which is defined by setting
\beq\la{ret-}
\xymatrix{
{^p\!D}^{\leq 0} (X,\oql)  = \{ K \, | \; i^* K \in {^p\!D}^{\leq 0} (X_s), \; j^* K \in {^p\!D}^{\leq -1}(X_\eta)\},  \\
{^p\!D}^{\geq 0} (X,\oql)  = \{ K \, | \; i^! K \in {^p\!D}^{\geq 0} (X_s), \; j^* K \in {^p\!D}^{\geq -1}(X_\eta)\}.
}
\eeq
If $X/S$ is obtained from a separated  scheme $Z$ of finite tpye over a smooth curve $C$ over a field, by localizing at a closed point on the curve $C$, then the usual middle perversity $t$-structure
on $Z$ induces the rectified one in (\ref{ret-}).

Note that $i_* {\oql}_s$,  ${\oql}_S [1]=R^0j_* {\oql}_{\eta}[1]$, $j_* {\oql}_\eta [1]$ and $j_! {\oql}_\eta [1]$ are rectified perverse.

{\bf Self-duality.}
By using the definition, and the fact that the middle perversity $t$-structure on a variety over a  field is self-dual
for the usual relative dualizing complexes  over the field, it is easy to verify that 
the rectified perverse $t$-structure is self-dual for the relative dualizing complex $K_{X/S}$, i.e. the duality functor
$D$ exchanges ${^p\!D}^{\leq 0} (X,\oql)$ and ${^p\!D}^{\geq 0} (X, \oql).$

{\bf $t$-exactness.}
The functors $j_*$ and $j_!$  and$(j^*[-1] = j^![-1]$, $\psi[-1]$ and $\phi$ are $t$-exact (\ci[\S4]{illusie}).
If $f:X \to Y$ is an  affine $S$-morphism,  and $S$ is a Henselian, then $f_*$ is right $t$-exact for the rectified perverse $t$-structure (\ci[Th. 2.4, due to O. Gabber]{illusie2}.
Let $f:X \to Y$ be a morphism. Let $d\geq 0$ be an integer such that every geometric fiber of $f$ has dimension at most $d.$
Then we have the following ``inequalities" for the rectified $t$-structure: 
 $f_!: {^p\!D}^{\leq 0} (X) \to {^p\!D}^{\leq d} (Y)$, 
$f^!:  {^p\!D}^{\geq 0} (Y) \to {^p\!D}^{\geq -d} (X) $,
$f^*: {^p\!D}^{\leq 0} (Y) \to {^p\!D}^{\leq d} (X)$,
$f_*: {^p\!D}^{\geq 0} (X) \to {^p\!D}^{\geq -d} (Y)$;
see  \ci[4.2.4]{bbdg} for the case of a field, that can be used to bootstrap the proof over $S$ by using (\ref{ret-}).

We  also have the analogues of \ci[4.1.10-11-12,  Prop. 4.2.5, 4.2.6]{bbdg}, which can be proved in the same way.

{\bf The category of rectified perverse sheaves.}
The category of rectified perverse sheaves on $X$ is Abelian, Noetherian, self-dual, Artinian, and every simple object
is an intermediate extension of a simple object (\ci[\S4.3]{bbdg}) from either the central, or the generic fiber (for this last item, see \ci[p.49, (c)]{illusie}).

\subsection{Compactification and specialization}\la{a cpt sp}$\;$

We refer to \S\ref{intro appdx}.
Let $C/B=J$ be a smooth curve as in \S\ref{rt51}.
The technical issues we need to address  are the following:

\ben
\item[(a)] 
The construction, of  a  suitable natural  completions of:   For  $J=\field$  an algebraically closed field$, M_{Hod}(C/\field)/\mathbb A^1_\field$
and of the associated Hodge-Hitchin morphism $M_{Hod} (C/k) \to A(\FC)\times \mathbb A^1_\field$. 
for $J=S$ a complete strictly Henselian DVR, of $M_{Dol}(C/B)/B$ and of the associated Hitchin morphism $M_{Dol}(C/B) \to A(C/B)).$
This way, diagram \ci[(70) and (72)]{de-2021} and their properties are in place (they are Cartesian up to nilpotents,
and this creates no problems when working with the \'etale topology).

\item[(b)] Taking the quotient by a possibly non-reduced flat finite group subscheme of $\GB$ in the proof of \ci[Lemma 4.1.1]{de-2021}.

\item[(c)] Being able to factor the $\Gm$ quotient into a  quotient as in (b), and a free quotient, as in  \ci[Lemma 4.1.1]{de-2021}.

\item[(d)]
The use of Luna slice Theorem in the proof of \ci[Lm. 4.1.1]{de-2021} in the context of a $\GB$ action  with trivial stabilizers on an affine variety;

\item[(e)] The use of Lemma \ci[Lm.\;4.1.4]{de-2021}, i.e. the assertion that if $p:A\to C$ is the quotient by a finite  flat group scheme $G$ over $\Bbase$ (with some extra assumptions to be listed in Lemma \ref{quot fgr 2}),
then ${\oql}_C$ is a direct summand of $p_*{\oql}_A$.

\item[(f)]  The use of Lemma \ci[Lm.\;4.1.3]{de-2021}, i.e. the identity $a^! \rat_X= \rat_Z[-2]$.
\een

Issue (a) is resolved by taking the compactifications in \S\ref{rt51}.

Issue (b) is resolved by 
the forthcoming standard Lemma \ref{quot fgr}.
We note that  this issue (b) can also be resolved if the positive characteristic $p$ of the ground field is bigger than the rank $r$  of the Higgs bundles we are taking: then the stabilizers we deal with are $i$-roots of unity for $i=1, \ldots, r$, and they are thus  reduced (finite cyclic).

Issue (c) is resolved in the forthcoming Lemma  \ref{quot fgr 3}.

Issue (d) is resolved by \cite[Thm.\;20.4]{al-ha-ry}, where the authors prove a relative version of the Luna Slice Theorem. 
In particular, loc.\;cit.\;implies that for a smooth affine $\Bbase$-scheme $X$ with a free action of a smooth affine reductive group scheme $G$ over $\Bbase$, if the GIT quotient $X/G$ exists, then etale locally over $\Bbase$, $X$ is etale locally isomorphic to the product of $G$ and a $\Bbase$-scheme $W,$ which, morally,  is  a slice 
 transversal to the $G$-orbit.

Issue (e) is resolved by the forthcoming Lemma \ref{quot fgr 2}.

Issue (f) is resolved by O. Gabber Purity result \ci[Tm. 2.2]{illusie2}.

\begin{lm}\la{quot fgr}
Let $X$ be a quasi-projective scheme over a noetherian base scheme $\Bbase$.
Let $G$ be a finite flat group scheme over $\Bbase$ that acts on $X$.
Then a uniform geometric quotient $q: X\ra X/G$ exists, the morphism $q$ is finite, and the quotient $X/G$ is quasi-projective over $\Bbase$.
\end{lm}
\begin{proof}

For the statements without the quasi-projectivity of $X/G$,
a proof is contained in \cite[Thm.\;4.16]{mo-va}. 
When $\Bbase$ is the spectrum of a field, a proof can also be found in \cite[Thm.\;12.1]{ma-mu-ra}. See also \cite[Rmk.\;4.2]{rydh2}.

The quasi-projectivity of $X/G$ over $B$ is proved by \cite[Prop.\;4.5.(${\rm B}'$)]{rydh2}
\end{proof}

Lemma \ref{quot fgr} implies the following Lemma \ref{quot fgr 3}, which is needed in the proof of Lemma \ref{quot fgr 2}.

\begin{lemma}
\label{quot fgr 3}
Let $X/\Bbase$ be as in Lemma \ref{quot fgr}.
Let $H$ be a group scheme over $\Bbase$.
Suppose that $X/\Bbase$ admits an $H$-action, so that the uniform geometric quotient $q:X\ra X/H$ exists.
Let $G$ be a finite flat closed group subscheme of $H$ so that the quotient group scheme $H/G$ exists and is reductive.
Then we have a factorizaiton
\[q: X\xra{q_1} X/G \xra{q_2} X/H,\]
where both $q_1$ and $q_2$ are uniform geometric quotients.
\end{lemma}
\begin{proof}
We first show that $H/G$ acts on $X/G$: 
We have the short exact sequence 
\[0\ra \mathcal{O}_{X}\otimes \mathcal{O}_{H/G}\ra \mathcal{O}_X\otimes O_H \ra \mathcal{O}_X\otimes \mathcal{O}_G\ra 0. \]
Since the composition $\mathcal{O}_{X/G}\hra \mathcal{O}_X\ra \mathcal{O}_X\otimes \mathcal{O}_H\twoheadrightarrow \mathcal{O}_X\otimes \mathcal{O}_G$, where the middle morphism is the action comorphism, is trivial,
we have a natural morphism $\mathcal{O}_{X/G}\ra \mathcal{O}_X\otimes \mathcal{O}_{H/G}$, which factors through
$\mathcal{O}_{X/G}\ra \mathcal{O}_{X/G}\otimes \mathcal{O}_{H/G}$. One can check that this defines the comorphism of an $H/G$ action on $X/G$.
 
By Lemma \ref{quot fgr}, we have that the uniform geometric quotient $q_1:X\ra X/G$ exists.
Using \cite[Thm.\;3,\;Rmk.\;8,10]{seshadri} as in the proof of Theorem \ref{ugm proper}.(2) in \S\ref{quotient ex}, we have that $q_2:X/G\ra (X/G)/(H/G)$ exists as a uniform categorical quotient, and that $q_2$ is a uniform geometric quotient if for any geometric point $a$ of $X/G$, the set-theoretic image $\mu_{a}(H/G)$ of the orbit morphism $\mu_a$ is closed. The closedness of $\mu_a(H/G)$ is shown by the proof of \cite[Lm.\;3.3.1.(1)]{bdhk}, thus $q_2$ is a uniform geometric quotient.
The factorization $q=q_1\circ q_2$ follows from the uniqueness of a geometric quotient.
\end{proof}

Recall that for any base scheme $B$ and any abstract group $G$, there exists a unique constant
group scheme over $B$ associated with $G$, see \cite[\href{https://stacks.math.columbia.edu/tag/03YW}{03YW}]{stacks}.

\begin{lm}\label{quot fgr 2}
Assume the setup as in Lemma \ref{quot fgr}.
Assume either one of the following additional assumptions on $\Bbase$ or $G$:
\ben
\item[(i)]
$G$ is the constant group scheme associated to an abstract finite group;
\item[(ii)]
$\Bbase=k$ is an algebraically closed field;
\item[(iii)]
$\Bbase$ is of equal characteristics and $G$ is the group scheme $\mu_N$ of $N$-th roots of unity for some $N\in \Z_{>0}$.
\een
Then $(\oql)_{X/G}$ is a direct summand of $q_*(\oql)_X$.
\end{lm}

\begin{proof} 
Let us start with the case with assumption (i).

Let $q:X/\ra X/G$ be the uniform geometric quotient as in Lemma \ref{quot fgr}.
We have that $(q_{*}(\oql)_X)^{G}\cong (\oql)_{X/G}$.
Since $char(\oql)=0$, there exists the trace morphism 
\[(s\mapsto \frac{1}{|G|}\sum_{g\in G} g^* s): q_*(\oql)_X \ra (q_{*}(\oql)_X)^G, \]
which defines a splitting of the inclusion $(\oql)_{X/G}\hra q_*(\oql)_X$.
The case with assumption (i) is then proved.

We now consider the case with assumption (ii): 
Let 
\[1\ra G_0\ra G\ra \pi_0(G)\ra 1\]
be the ``connected-\'etale short exact sequence" as in \cite[p.114]{milne}, i.e. $G_0$ is the unique connected normal subgroup scheme of $G$ so that $\pi_0(G):=G/G_0$ is etale over $\field$. 
Recall that  $G$ is reduced if and only if  $G_0$ is trivial. 

By Lemmata \ref{quot fgr} and \ref{quot fgr 3}, the uniform geometric quotients $q_1: X\ra X/G_0$, $q_2: (X/G_0)/\pi_0(G)$ and $q:X\ra X/G$ exist, and we have a factorization 
\[q:X\xra{q_1} X/G_0\xra{q_2} (X/G_0)/\pi_0(G).\]

Since $G_0$ is connected and $k$ is algebraically closed, by \cite[\href{https://stacks.math.columbia.edu/tag/054N}{054N}]{stacks} we have that for any field extension $K\supset k$, $(G_0)_K$ is also connected. 
Since $|(G_0)_K|$ is discrete, we have that $|(G_0)_K|$ is a singleton.
Since $(q_1)_K:X_K\ra (X/G_0)_K=X_K/(G_0)_K$ is also a geometric quotient, we have that $(q_1)_K$ is injective.
By \cite[\href{https://stacks.math.columbia.edu/tag/01S4}{0154}]{stacks}, we have that $q_1$ is universally injective, and thus purely inseparable.
Since $q_1$ is also finite and surjective by Lemma \ref{quot fgr}, we have that $q_1$ is a universal homeomorphism \cite[\href{https://stacks.math.columbia.edu/tag/04DF}{04DF}]{stacks}.
Therefore we have that $q_{1,*}(\oql)_X=(\oql)_{X/G_0}$, see \cite[Rmk.\;2.3.17]{etcoh}.

Therefore, to show that $(\oql)_{X/G_0}$ is a direct summand of $q_*(\oql)_X$, it suffices to show that 
$(\oql)_{X/G_0}$ is a direct summand of $q_{2,*}(\oql)_{X/G_0}$.
Note that $\pi_0(G)$ is associated with an abstract finite group, see e.g. \cite[\S 8.21]{jantzen}. 
We are then reduced to the case with assumption (i). The case with assumption (ii) is thus finished.

For the case with assumption (iii), \cite[Lm.\;19.7]{al-ha-ry} shows that there exists 
a subgroup scheme $G_0$ of $\mu_N$ so that fiber by fiber over $\Bbase$, $G_0$ restricts to the identity component of $\mu_N$, and that the quotient group scheme $G_0/\mu_N$ is etale over $\Bbase$.
Therefore the short exact sequence 
\[1\ra G_0\ra \mu_N\ra \mu_N/G_0\ra 1,\]
fiber by fiber over $\Bbase$, restricts to the \'{e}tale-connected sequence in the proof in case (ii) above.
Moreover, using the fact that $\mu_N$ is cyclic and $\Bbase$ is of equal characteristic, we see that 
$\mu_N/G_0$ is isomorphic to the group scheme associated with the abstract finite group $(\mu_N)_b(\kappa(b))$ 
for some closed point $b$ in $B$ with residue field $\kappa(b)$.
Therefore, similar argument as in case (ii) finishes the proof in case (iii).
\end{proof}

\end{document}